\documentclass[]{interact}

\usepackage{epstopdf}
\usepackage[caption=false]{subfig}

\usepackage[numbers,sort&compress]{natbib}
\bibpunct[, ]{[}{]}{,}{n}{,}{,}
\makeatletter
\def\NAT@def@citea{\def\@citea{\NAT@separator}}
\makeatother

\theoremstyle{plain}
\newtheorem{theorem}{Theorem}[section]
\newtheorem{lemma}[theorem]{Lemma}

\newtheorem{definition}[theorem]{Definition}

\theoremstyle{remark}
\newtheorem{remark}[theorem]{Remark}

\usepackage{amssymb}
\usepackage{amsmath}
\usepackage{graphicx}
\newcommand{\be}{\begin{equation}}
\newcommand{\ee}{\end{equation}}
\newcommand{\bea}{\begin{eqnarray*}}
\newcommand{\eea}{\end{eqnarray*}}
\newcommand{\ba}{\begin{array}}
\newcommand{\ea}{\end{array}}
\newcommand{\ei}{\end{itemize}}
\newcommand{\bc}{\begin{center}}
\newcommand{\ec}{\end{center}}
\newcommand{\bfr}{\begin{flushright}}
\newcommand{\efr}{\end{flushright}}

\numberwithin{equation}{section}

\usepackage{graphicx}
\usepackage{latexsym}
\usepackage[colorlinks]{hyperref}

 \numberwithin{equation}{section}

\newcommand{\comment}[1]{}
\allowdisplaybreaks

\begin{document}

\title{Potentials and transmission problems in weighted Sobolev spaces for anisotropic Stokes and Navier-Stokes systems with $L_{\infty}$ strongly elliptic coefficient tensor}

\author{
\name{Mirela Kohr\textsuperscript{a}\thanks{CONTACT Mirela Kohr. Email: mkohr@math.ubbcluj.ro} and Sergey E. Mikhailov\textsuperscript{b} and Wolfgang L. Wendland\textsuperscript{c}}
\affil{\textsuperscript{a}Faculty of Mathematics and Computer Science, Babe\c s-Bolyai University,
Cluj-Napoca, Romania;
\textsuperscript{b}Department of Mathematics, Brunel University London, Uxbridge,
UK;
\textsuperscript{c}Institut f\"ur Angewandte Analysis und Numerische Simulation, Universit\"at Stuttgart,
Stuttgart, Germany
}
}

\maketitle

\begin{abstract}
We obtain well-posedness results in $L_p$-based weighted Sobolev spaces for a transmission problem for anisotropic Stokes and Navier-Stokes systems with $L_{\infty}$ strongly elliptic coefficient tensor, in complementary Lipschitz domains of ${\mathbb R}^n$, $n\ge 3$. 
The strong ellipticity allows to explore the associated pseudostress setting.
First, we use a variational approach that reduces two linear transmission problems for the anisotropic Stokes system to equivalent mixed variational formulations with data in $L_p$-based weighted Sobolev and Besov spaces. We show that such a mixed variational formulation is well-posed in the space ${\mathcal H}^1_{p}({\mathbb R}^n)^n\times L_p({\mathbb R}^n)$, {$n\geq 3$}, for any $p$ in an open interval containing $2$.
These results are used to define the Newtonian and layer potential operators for the considered anisotropic Stokes system.
Various mapping properties of these operators are also obtained. 
The potentials are employed to show the well-posedness of some linear transmission problems, which then is combined with a fixed point theorem in order to show the well-posedness of the nonlinear transmission problem for the anisotropic Stokes and Navier-Stokes systems in $L_p$-based weighted Sobolev spaces, whenever the given data are small enough.
\end{abstract}


\begin{quote}
\end{quote}


\section{Introduction}
\label{Intro}
\setcounter{equation}{0}
A powerful tool in the analysis of boundary value problems for partial differential equations is played by the layer potential methods.
Mitrea and Wright \cite{M-W} used them
to obtain well-posedness results for the main boundary value problems for the constant-coefficient Stokes system in
Lipschitz domains in ${\mathbb R}^n$  in Sobolev, Bessel potential, and Besov spaces (see also
\cite[Proposition 4.5]{B-H} for an unsteady exterior Stokes problem). The authors in \cite{K-L-M-W} obtained mapping properties of the constant-coefficient Stokes and Brinkman layer potential operators in standard and weighted Sobolev spaces by exploiting results of singular integral operators (see also \cite{K-L-W,K-L-W1}).

The methods of layer potential theory play also a significant role in the study of elliptic boundary problems with variable coefficients. Mitrea and Taylor \cite[Theorem 7.1]{M-T} used the technique of layer potentials to prove the well-posedness of the Dirichlet problem for the Stokes system in $L_p$-spaces on arbitrary Lipschitz domains in a compact Riemannian manifold. Dindo\u{s} and Mitrea
\cite[Theorems 5.1, 5.6, 7.1, 7.3]{D-M} used a boundary integral approach to show well-posedness results in Sobolev and Besov spaces for Poisson problems of Dirichlet type for the Stokes and Navier-Stokes systems with smooth coefficients in Lipschitz domains on compact Riemannian manifolds.
A layer potential analysis of pseudodifferential operators of Agmon-Douglis-Nirenberg type in Lipschitz domains on compact Riemannian manifolds has been developed in \cite{10-new2}. The authors in \cite{K-M-W} used a layer potential approach and a fixed point theorem to show well-posedness of transmission problems for the Navier-Stokes and Darcy-Forchheimer-Brinkman systems with smooth coefficients in Lipschitz domains on compact Riemannian manifolds.
Choi and Lee \cite{Choi-Lee} proved the well-posedness in Sobolev spaces for the Dirichlet problem for the Stokes system with non-smooth coefficients in a Lipschitz domain $\Omega \subset {\mathbb R}^n$ ($n\geq 3$) with a small Lipschitz constant when the coefficients have vanishing mean oscillations (VMO) with respect to all variables. 
Choi and Yang \cite{Choi-Yang} established existence and pointwise bound of the fundamental solution for the Stokes system with measurable coefficients in the space ${\mathbb R}^d$, $d\geq 3$, when the weak solutions of the system are locally H\"{o}lder continuous.

Alliot and Amrouche \cite{Al-Am} developed a variational approach to show the existence of weak solutions for the exterior Stokes problem in weighted Sobolev spaces (see also \cite{Amrouche-1,Gi-Se}).
The authors in \cite{K-W} developed a variational approach in order to analyze Stokes and Navier-Stokes systems with $L_\infty $ coefficients in Lipschitz domains on compact Riemannian manifolds (see also \cite{K-W1}).

An alternative integral approach, which reduces various boundary value problems for variable-coefficient elliptic partial differential equations to {\em boundary-domain integral equations} (BDIEs), by means of explicit parametrix-based integral potentials, was explored e.g., in \cite{CMN-1,CMN-2,Ch-Mi-Na-3, Mikh-18}.
Equivalence of BDIEs to the boundary problems and invertibility of BDIE operators in $L_2$ and $L_p$-based Sobolev spaces have been analyzed in these works.
Localized boundary-domain integral equations based on a harmonic parametrix for divergence-form elliptic PDEs with variable matrix coefficients have been {also developed, see \cite{CMN2017}
and the references therein.}

Brewster et al. in \cite{B-M-M-M} used a variational approach to show well-posedness results for Dirichlet, Neumann and mixed problems for higher order divergence-form elliptic equations with $L_{\infty }$ coefficients in locally $(\epsilon,\delta )$-domains and in Besov and Bessel potential spaces. Sayas and Selgas in \cite{Sa-Se} developed a variational approach for the constant-coefficient Stokes layer potentials, by using the technique of N\'{e}d\'{e}lec \cite{Ne}. 
B\u{a}cu\c{t}\u{a}, Hassell and Hsiao \cite{B-H} developed a variational approach for the constant-coefficient Brinkman single layer potential and analyzed the time-dependent exterior Stokes problem with Dirichlet condition in ${\mathbb R}^n$, $n\!=\!2,3$. Barton \cite{Barton} used the Lax-Milgram Lemma to construct layer potentials for {strongly} elliptic operators in general settings.

{Throughout this paper, we use the Einstein convention on summation in repeated indices from $1$ to $n$,} and the standard notation $\partial _\alpha $ for the first order partial derivative with respect to the variable ${x_\alpha }$, $\alpha =1,\ldots ,n$.
{Let $\check{\boldsymbol{\mathcal L}}$ be a second order differential operator in divergence form,}
\begin{equation}
\label{Stokes-0}
\begin{array}{lll}
{\check{\boldsymbol{\mathcal L}}{\bf u}:=\partial _\alpha\left(A^{\alpha \beta }\partial _\beta {\bf u}\right),}
\end{array}
\end{equation}
where ${\mathbb A}=\left\{A^{\alpha \beta }\right\}_{1\leq \alpha ,\beta \leq n}$ is the viscosity coefficient fourth order tensor,  and for fixed $\alpha$ and $\beta$ $A^{\alpha \beta }=A^{\alpha \beta }(x)$ are $n\times n$ matrix-valued functions on ${\mathbb R}^n$, such that
\begin{align}
\label{Stokes-1}
A^{\alpha \beta }=\left\{a_{ij}^{\alpha \beta }\right\}_{1\leq i,j\leq n},\ \
a_{ij}^{\alpha \beta }\in L_{\infty }({\mathbb R}^n),\ 1\leq i,j,\alpha ,\beta \leq n.
\end{align}
We will further shorten \eqref{Stokes-1} as $\mathbb A\in  L_{\infty }({\mathbb R}^n)^{n^4}$.
We assume that the {\em boundedness condition} $\big|a_{ij}^{\alpha \beta }(x)\big|\leq c_{\mathbb A}$ and the {\it strong ellipticity condition}
\begin{align}
\label{mu}
a_{ij}^{\alpha \beta }(x)\xi _{i\alpha }\xi _{j\beta }\geq c_{\mathbb A}^{-1}\xi _{i\alpha }\xi _{i\alpha}
=c_{\mathbb A}^{-1}|\boldsymbol \xi|^2\quad
\forall\, \boldsymbol\xi \!=\!(\xi _{i\alpha })_{1\leq i,\alpha \leq n}\!\in \! {\mathbb R}^{n\times n}
\end{align}
hold for almost any $x\in {\mathbb R}^n$, with a constant $c_{\mathbb A} >0$ (cf. \cite[(7.23)]{B-M-M-M}, \cite[(1.1)]{Choi-Dong-Kim-JMFM}).

Let ${\bf u}$ be an unknown vector field for velocity, $\pi $ be an unknown scalar field for pressure, and ${\bf f}$ be a given vector field for distributed forces, defined on an open set ${\mathfrak D}\subset {\mathbb R}^n$ with the compact boundary $\partial{\mathfrak D}$. Then the equations
\begin{equation}
\label{Stokes}
\begin{array}{lll}
\boldsymbol{\mathcal L}({\bf u},\pi ):=\partial_\alpha\left(A^{\alpha \beta }\partial_\beta {\bf u}\right)
-\nabla \pi={\bf f},\ {\rm{div}}\ {\bf u}=0 \mbox{ in } {\mathfrak D}
\end{array}
\end{equation}
determine the {\it Stokes system} with $L_{\infty}$ tensor viscosity coefficient.

Let $\lambda\in L_{\infty }({\mathbb R}^n)$. Then the nonlinear system
\begin{equation}
\label{Navier-Stokes}
\begin{array}{lll}
{\partial_\alpha\left(A^{\alpha \beta }\partial _\beta {\bf u}\right)-\lambda ({\bf u}\cdot \nabla ){\bf u}-\nabla \pi={\bf f},\ {\rm{div}}\, {\bf u}=0 \mbox{ in } {\mathfrak D},}
\end{array}
\end{equation}
is called the {{\it anisotropic Navier-Stokes system}} with $L_{\infty}$ viscosity tensor ${\mathbb A}=\left(A^{\alpha \beta }\right)_{1\leq \alpha ,\beta \leq n}$.

The systems \eqref{Stokes} and \eqref{Navier-Stokes} can describe flows of viscous incompressible fluids with {\it anisotropic viscosity tensor}, and the viscosity tensor ${\mathbb A}$ is related to the physical properties of such a fluid (see \cite{duffy,Choi-Dong-Kim-JMFM,Ni-Be}).
{\it Our goal is to treat transmission problems for the Stokes and Navier-Stokes systems \eqref{Stokes} and  \eqref{Navier-Stokes} in ${\mathbb R}^n\setminus\partial \Omega $, where $\partial \Omega $ is a Lipschitz boundary}. Then we have to add 
adequate conditions at infinity by setting our problems in weighted Sobolev spaces.
\begin{remark}
\label{example-1}
In the isotropic case
\begin{align}
\label{isotropic}
{\bar a_{ij}^{\alpha\beta}=\mu \left(\delta_{\alpha j}\delta _{\beta i}+\delta_{\alpha \beta }\delta _{ij}\right),\ 1\leq i,j, \alpha ,\beta \leq n}
\end{align}
(see \cite{duffy}), with $\mu\in L_{\infty }({\mathbb R}^n)$, we assume that there exists a constant $c_\mu >0$, such that
$c_\mu^{-1}\leq\mu \leq {c_\mu} \mbox{ a.e. in } {\mathbb R}^n.$
In such a case, the operator $\boldsymbol{\mathcal L}$ given by \eqref{Stokes} takes the form
\begin{align}
\label{Stokes-mu}
\boldsymbol{\mathcal L}({\bf u},\pi)={\rm{div}}\left(\mu \nabla {\bf u}\right )-\nabla \pi
\end{align}
if ${\rm{div}}\, {\bf u}=0$.
{The tensor $\bar a_{ij}^{\alpha\beta}$ given by \eqref{isotropic} satisfies the second (ellipticity) condition in \eqref{mu} only for symmetric matrices $\boldsymbol\xi$.}
On the other hand, for any ${\bf u}$ and $\pi$, $\boldsymbol{\mathcal L}({\bf u},\pi)$ given by \eqref{Stokes-mu}
can be also represented as
\begin{align}
\label{isotropic-1}
{\mathcal L}_i({\bf u},\pi)
=\partial_\alpha(a_{ij}^{\alpha \beta}\partial_\beta u_j)-\partial_i\pi,\quad
a_{ij}^{\alpha \beta}=\mu \delta_{\alpha \beta }\delta _{ij},\ 1\leq i,j, \alpha ,\beta \leq n,
\end{align}
where
${a_{ij}^{\alpha \beta }(x)\xi _{i\alpha }\xi _{j\beta }=\mu (x)\xi _{i\alpha }\xi _{i\alpha }\geq 2c_\mu ^{-1}|\boldsymbol \xi |^2,}$
for a.e. $x\in {\mathbb R}^n$ and for any $\boldsymbol\xi =(\xi _{i\alpha })_{1\leq i,\alpha \leq n}\in {\mathbb R}^{n\times n}$.
Hence the ellipticity condition \eqref{mu} is satisfied for any matrices, and
our analysis is also applicable to the {\it isotropic Stokes system}.
Note that $a_{ij}^{\alpha \beta}\partial_\beta u_j=\mu\partial_\alpha u_i$ can be associated with the viscous part of the pseudostress $\mu\partial_\alpha u_i-\delta_{\alpha i}\pi$, cf., e.g., \cite{Cai-Wang-2010}.
The approaches based on the pseudostress formulation have been intensively used in the study of viscous incompressible fluid flows due to their ability to avoid the symmetry condition that appears in the approaches based on the standard stress formulation (see, e.g., \cite{Cai-Wang-2010,Camano-Gatica2016}).
\end{remark}

\section{Preliminary results}\label{S2}

Let further on in the paper $\Omega _{+}:=\Omega $  be a bounded Lipschitz domain in ${\mathbb R}^n$ ($n\geq 3$) with connected boundary $\partial \Omega $. Let $\Omega _{-}:={\mathbb R}^n\setminus \overline{{\Omega }}_{+}$. Let $\mathring E_\pm$ denote the operator of extension by zero outside $\Omega_\pm$.

\subsection{Standard $L_p$-based Sobolev spaces and related results}
For $p\in (1,\infty )$, $L_p({\mathbb R}^n)$ denotes the Lebesgue space of (equivalence classes of) measurable, $p^{{\rm th}}$ integrable functions on ${\mathbb R}^n$, and $L_{\infty }({\mathbb R}^n)$ denotes space of (equivalence classes of) essentially bounded measurable functions on ${\mathbb R}^n$.
{For any $p\in (1,\infty )$, the conjugate exponent $p'$ is given by $\frac{1}{p}+\frac{1}{p'}=1$.}
Given a Banach space ${\mathcal X}$, its topological dual is denoted by ${\mathcal X}'$. The duality pairing of two dual spaces defined on a subset $X\subseteq {\mathbb R}^n$ is denoted by $\langle \cdot ,\cdot \rangle _X$.
Let $H^1_{p}({\mathbb R}^n)$ and $H^1_{p}({\mathbb R}^n)^n$ denote the standard $L_p$-based Sobolev (Bessel potential) spaces.

For any open set $\Omega '$ in $\mathbb R^n$, let 
${\mathcal D}(\Omega '):=C^{\infty }_{0}(\Omega ')$ denote the space of infinitely differentiable functions with compact support in $\Omega '$, equipped with the inductive limit topology. 
Let 
${\mathcal D}'(\Omega ')$ denote the corresponding space of distributions on $\Omega '$, i.e., the dual space of ${\mathcal D}(\Omega ')$. 
Let
$H^1_{p}({\Omega '}):=\{f\in {\mathcal D}'(\Omega '):\exists \, F\in H^1_{p}({\mathbb R}^n)
\mbox{ such that } F_{|{\Omega '}}=f\},$
where $_{|\Omega '}$ denotes the restriction operator onto $\Omega '$.
The space $\widetilde{H}^1_{p}({\Omega '})$ is the closure of ${\mathcal D}(\Omega ')$ in $H^1_{p}({\mathbb R}^n)$.
Also, $H^1_{p}({\Omega '})^n$ and $\widetilde{H}^1_{p}({\Omega '})^n$ are the spaces of vector-valued functions with components in $H^1_{p}({\Omega '})$ and $\widetilde{H}^1_{p}({\Omega '})$, respectively, {and similar extensions to the vector-valued functions or distributions are assumed to all other spaces introduced further}.
The Sobolev space $\widetilde{H}^1_{p}({\Omega '})$ can be identified with the closure $\mathring{H}^1_{p}(\Omega ')$ of ${\mathcal D}(\Omega ')$ in ${H}^1_{p}({\Omega '})$ (see, e.g., \cite{J-K1}, and \cite[Theorem 3.33]{Lean} for $p=2$). 
For $p\!\in \!(1,\infty)$ and $s\!\in \!(0,1)$, the boundary Besov space $B^s_{p,p}(\partial \Omega )$ can be defined by means of the method of real interpolation,
$B^s_{p,p}(\partial \Omega )=\left(L_p(\partial \Omega ),H^1_{p}(\partial \Omega )\right)_{s,p}$
(cf., e.g., \cite[Chapter 1, \& 1.3]{Triebel}, \cite[Section 11.1]{M-W}). 
The dual of $B^s_{p,p}(\partial \Omega )$ is the space $B^{-s}_{p',p'}(\partial \Omega )$. 
For $p\!=\!2$, we use the standard notation for the $L_2$-based Sobolev spaces
$
H^1(\Omega ')^n=H^1_{2}(\Omega '),\,
H^s(\partial \Omega )=H^s_{2}(\partial \Omega )^n=B_{2,2}^s(\partial \Omega ).
$
For further properties of standard Sobolev and Besov spaces we refer the reader to \cite{J-K1,Lean,M-W,Triebel}.

We often use the {following result} (see \cite{Co}, \cite[Lemma 2.6]{Mikh}, \cite[Theorem 2.5.2]{M-W}).
\begin{lemma}
\label{trace-operator1}
Let $\Omega _{+}$ be a bounded Lipschitz domain of ${\mathbb R}^n$  with connected boundary $\partial \Omega $,
and let $\Omega _{-}\!:=\!{\mathbb R}^n\setminus \overline{\Omega }$ be the corresponding
exterior domain.
If $p\!\in \!(1,\infty )$, then there exist a linear bounded trace operator $\gamma_{\pm }\!:\!H^s_{p}({\Omega }_{\pm })\!\to \!B^{s-\frac{1}{p}}_{p,p}({\partial\Omega })$ such that $\gamma_{\pm }f\!=\!f_{|{{\partial\Omega }}}$ for any $f\!\in \!C^{\infty }(\overline{\Omega }_{\pm })$.
The operator $\gamma _{\pm }$ is surjective and has a $($non-unique$)$ linear and bounded right inverse $\gamma ^{-1}_{\pm }\!:\!B^{1-\frac{1}{p}}_{p,p}({\partial\Omega })\!\to \!H^s_{p}({\Omega }_{\pm }).$ The trace operator $\gamma :{H}^1_{p}({\mathbb R}^n)\to B_{p,p}^{1-\frac{1}{p}}(\partial \Omega )$ is also well defined and bounded.
\end{lemma}

\subsection{Weighted Sobolev spaces}
\label{S2.2}
Given $n\in {\mathbb N}$, $n\geq 3$, let $\rho :{\mathbb R}^n\to {\mathbb R}_+$ denote the weight function
\begin{align}
\label{rho}
\rho ({\bf x})=(1+|{\bf x}|^2)^{\frac{1}{2}}\,.
\end{align}
Let $p\in (1,\infty )$ and $\lambda \in {\mathbb R}$. Then the weighted Lebesgue space $L_p(\rho^\lambda ;{\mathbb R}^n)$ is defined as
\begin{align}
\label{Lp-weight}
&f\in {L_{p}(\rho^\lambda ;{\mathbb R}^n)} \Longleftrightarrow {\rho }^\lambda f\in L_{p}({\mathbb R}^n),
\end{align}
and
$L_{2}(\rho^\lambda ;{\mathbb R}^n)$
is a Hilbert space. 
We also consider the weighted Sobolev space ${\mathcal H}^1_{p}({\mathbb R}^n)$
(cf. \cite[Definition 1.1]{Al-Am}, \cite[Theorem I.1]{Ha}) consisting of functions $f$, for which the norm $\|f\|_{{\mathcal H}^1_{p}({\mathbb R}^n)}$, defined by
\begin{align}
\label{weight-2p}
\|f\|_{{\mathcal H}^1_{p}({\mathbb R}^n)}^p:=
\left\{\begin{array}{ll}
\big\|{\rho }^{-1}f\big\|_{L_p({\mathbb R}^n)}^p+\|\nabla f\|_{L_p({\mathbb R}^n)^n}^p & \mbox{ if } p\neq n\,,\\
\big\|{\rho }^{-1}\big(\ln (1+\rho ^2)\big)^{-1}f\big\|_{L_p({\mathbb R}^n)}^p+\|\nabla f\|_{L_p({\mathbb R}^n)^n}^p & \mbox{ if } p=n\,,
\end{array}
\right.
\end{align}
is bounded. This is a reflexive Banach space. The space ${\mathcal H}^{-1}_{p'}({\mathbb R}^n)$ is defined as the dual of the space ${\mathcal H}^1_{p}({\mathbb R}^n)$.

For the functions from ${\mathcal H}^1_{p}({\mathbb R}^n)$, the semi-norm
\begin{align}
\label{seminorm-R3}
|f|_{{\mathcal H}^1_{p}({\mathbb R}^n)}:=\|\nabla f\|_{L_p({\mathbb R}^n)^n}
\end{align}
is equivalent to the norm $\|\cdot \|_{{\mathcal H}^1_{p}({\mathbb R}^n)}$, given by \eqref{weight-2p}, if $1<p<n$ 
(cf., e.g., \cite[Theorem 1.1]{Al-Am-1}). Consequently,
\begin{align}
\label{Kozono}
{{\mathcal H}^1_{p}({\mathbb R}^n)={\hat{H}}^1_{p;0}({\mathbb R}^n)}
\end{align}
for $1<p<n$, where ${\hat{H}}^1_{p;0}({\mathbb R}^n)$ is the closure of the space ${\mathcal D}({\mathbb R}^n)$ with respect to the semi-norm \eqref{seminorm-R3}, cf. \cite[Proposition 2.4]{Kozono-Shor}.
Hence, the space ${\mathcal D}({\mathbb R}^n)$ is dense in ${\mathcal H}^1_{p}({\mathbb R}^n)$ (cf., e.g., \cite{Al-Am,Ha}).
Moreover, {for this range of $p$,
\begin{align}
\label{Kozono-1}
{\hat{H}}^1_{p;0}({\mathbb R}^n)=\left\{u\in L_{\frac{np}{n-p}}({\mathbb R}^n):
\nabla u\in L_p({\mathbb R}^n)^n\right\},
\end{align}
and the divergence operator ${\rm{div}}\!:\!{\hat{H}}^1_{p;0}({\mathbb R}^n)^n\!\to \!L_p({\mathbb R}^n)$ is surjective (cf. \cite[Proposition 2.4 (i), Lemma 2.5]{Kozono-Shor}).}

The set $\{{\mathcal H}^1_{p}({\mathbb R}^n)\}_{1<p<n}$ is a complex interpolation scale, which means that
\begin{align}
\label{int-weight}
[{\mathcal H}^1_{p_1}({\mathbb R}^n),{\mathcal H}^1_{p_2}({\mathbb R}^n)]_\theta ={\mathcal H}^1_{p}({\mathbb R}^n)\,,
\end{align}
whenever $p_1,p_2\!\in \!(1,n)$, $\theta \!\in \!(0,1)$, and $\frac{1}{p}\!=\!\frac{1-\theta}{p_1}\!+\!\frac{\theta }{p_2}$ (see 
\cite[Theorem 3]{Triebel-2}, \cite[Theorem 2.1, Corollary 2.7]{Lang-Mendez}). By $[\cdot ,\cdot ]_\theta $
we denote the space
obtained with the complex interpolation method, and the equality of spaces in \eqref{int-weight} holds with equivalent norms.
The complex interpolation spaces backgrounds can be found, e.g., in \cite[Chapter 4]{B-L} and \cite[Section 1.9]{Triebel}.

The space ${\mathcal H}^1_{p}(\Omega _{-})$ can be defined in terms of the norm $\|\cdot \|_{{\mathcal H}^1_{p}(\Omega _{-})}$, which has a similar expression to the norm in \eqref{weight-2p}, but with $\Omega _{-}$ in place of ${\mathbb R}^n$, and is a reflexive Banach space. The space $\widetilde{\mathcal H}^{-1}_{p'}(\Omega _{-})$ is defined as the dual of the space ${\mathcal H}^1_{p}(\Omega _{-})$.

Let $\mathring{\mathcal H}^1_{p}(\Omega _{-})\subset {\mathcal H}^1_{p}(\Omega_-)$ denote the closure of the space ${\mathcal D}({\Omega  }_{-})$ in ${\mathcal H}^1_{p}(\Omega _{-})${, and let
$\widetilde{\mathcal H}^1_{p}(\Omega _{-})\subset {\mathcal H}^1_{p}(\mathbb R^n)$ denote the closure of the space ${\mathcal D}({\Omega  }_{-})$ in ${\mathcal H}^1_{p}(\mathbb R^n)$.
The space $\widetilde{\mathcal H}^1_{p}(\Omega _{-})$ can be also characterised as
\begin{align}
\widetilde{\mathcal H}^1_{p}(\Omega _{-})=\left\{u\in {\mathcal H}^1_{p}(\mathbb R^n):{\rm{supp}}\, {u}\subseteq \overline{\Omega }_{-}\right\},
\end{align}
and identifies isomorphically with $\mathring{\mathcal H}^1_{p}(\Omega _{-})$ via the operator $\mathring E_{-}$ of extension by zero outside $\Omega _{-}$ (see, e.g., \cite[(2.9)]{B-M-M-M}).
The space ${\mathcal H}^{-1}_{p'}(\Omega _{-})$ is defined as the dual of the space $\widetilde{\mathcal H}^1_{p}(\Omega _{-})$. Since ${\mathcal D}(\Omega _-)$ is dense in $\mathring{\mathcal H}^1_{p}(\Omega _{-})$, and in $\widetilde{\mathcal H}^1_{p}(\Omega _{-})$, ${\mathcal H}^{-1}_{p'}(\Omega _{-})$ is a space of distributions.}

The space $\mathring{\mathcal H}^1_{p}(\Omega _{-})$ can be characterized as
\begin{equation}
\label{property}
\mathring{\mathcal H}^1_{p}(\Omega _{-})=\big\{v\in {{\mathcal H}^1_{p}}(\Omega _{-}):\gamma_{-}v=0 \mbox{ on } \partial \Omega \big\}
\end{equation}
(cf., e.g., \cite[(1.2)]{AGG1997}, \cite[Theorem 4.2, (4.16)]{B-M-M-M}).

For $p\in (1,n)$, the semi-norm
\begin{align}
\label{seminorm}
|f|_{{\mathcal H}^1_{p}(\Omega _{-})}:=\|\nabla f\|_{L_p(\Omega _{-})^n}
\end{align}
is a norm on the space ${\mathcal H}^1_{p}(\Omega _{-})$ that is equivalent to the full norm $\|\cdot \|_{{\mathcal H}^1_{p}(\Omega _{-})}$ given by \eqref{weight-2p} with $\Omega _{-}$ in place of ${\mathbb R}^n$. Moreover, the semi-norm \eqref{seminorm} is an equivalent norm on the space $\mathring{\mathcal H}^1_{p}(\Omega _{-})$ for any $p\in (1,\infty )$ (cf., e.g., \cite[Theorem 1.2]{AGG1997}, \cite[Theorem 1.2]{Al-Am}).
Consequently,
$
\mathring{\mathcal H}^1_{p}(\Omega _{-})={\hat{H}}^1_{p;0}(\Omega _{-}),$
for any $p\in (1,\infty )$, where ${\hat{H}}^1_{p;0}(\Omega _{-})$ is the closure of ${\mathcal D}(\Omega _{-})$ in the semi-norm \eqref{seminorm} (cf., e.g., \cite[Remark 1.3]{Al-Am}).

In addition, the statement of Lemma \ref{trace-operator1} extends to the space ${\mathcal H}^1_{p}({\Omega }_{-})$. Hence, there is a bounded, surjective {\it exterior trace operator}
\begin{align}
\label{ext-trace}
\gamma_{-}:{\mathcal H}^1_{p}({\Omega }_{-})\to B_{p,p}^{1-\frac{1}{p}}({\partial\Omega })
\end{align}
{(see, e.g., {\cite[p. 69]{Sa-Se}})}. Moreover, 
there exists a (non-unique) linear bounded right inverse $\gamma^{-1}_-:B_{p,p}^{1-\frac{1}{p}}({\partial\Omega })\!\to \! \mathcal{H}^1_{p}({\Omega }_-)$ of operator \eqref{ext-trace} (see \cite[Lemma 2.2]{K-L-M-W}).
 \cite[p. 1350006-4]{Ch-Mi-Na-3}).
The trace operator $\gamma \!:\!{\mathcal H}^1_{p}({\mathbb R}^n)\!\to \! B_{p,p}^{1-\frac{1}{p}}({\partial\Omega })$ is also linear, bounded and surjective $($cf., e.g., \cite[Theorem 2.3, Lemma 2.6]{Mikh}, \cite[(2.2)]{B-H} for $p=2$$)$.

In the case $p=2$, we employ the notations
$
{\mathcal H}^{\pm 1}({\mathbb R}^n):={\mathcal H}^{\pm 1}_{2}({\mathbb R}^n)$,
${\mathcal H}^{\pm 1}(\Omega _{-}):={\mathcal H}^{\pm 1}_{2}(\Omega _{-})
$, $H^s(\partial \Omega )=B_{2,2}^s(\partial \Omega )$,
and note that all these spaces are Hilbert spaces.

For $1<p<n$, let us also introduce the space $\mathcal H^1_p(\mathbb R^n\setminus\partial\Omega)$ consisting of functions $u$, for which the norm
\begin{align}
\label{standard-weight-p}
\|{u}\|_{\mathcal H^1_p(\mathbb R^n\setminus\partial\Omega)}
=\left(\|\rho ^{-1}{u}\|_{L_p({\mathbb R}^n)}^p
+\|\nabla {u}\|_{L_p(\Omega_+\cup\,\Omega_-)^{n}}^p\right)^{\frac{1}{p}}
\end{align}
is bounded.
Evidently, then \mbox{$u|_{\Omega_+}\in {H}^1_{p}(\Omega _+)$}, 
\mbox{$u|_{\Omega_-}\in  {\mathcal H}^1_{p}(\Omega _{-})$}, 
and the norm
$\left(\|u\|^p_{H^1_p(\Omega_+)}+\|u\|^p_{\mathcal H^1_p(\Omega_-)}\right)^{\frac{1}{p}}$ is equivalent to the norm \eqref{standard-weight-p} in $\mathcal H^1_p(\mathbb R^n\setminus\partial\Omega)$.
The jump of $u$ across $\partial \Omega $ is given by $\left[{\gamma }(u)\right]:={\gamma }_{+}(u)-{\gamma }_{-}(u)$. 
If ${u}\in {\mathcal H}^1_{p}({\mathbb R}^n\setminus \partial \Omega )$ and $\left[{\gamma }(u)\right]=0$ then ${u}\in {\mathcal H}^1_{p}({\mathbb R}^n)$, and conversely, if ${u}\in {\mathcal H}^1_{p}({\mathbb R}^n)$, then $\left[{\gamma }(u)\right]=0$ (cf., e.g., \cite[Theorem 5.13]{B-M-M-M}).

\begin{remark}
\label{Leray}
Let $B_R$ denote the ball of radius $R$ in ${\mathbb R}^n$ and center at the origin (assumed to be a point of the bounded Lipschitz domain $\Omega $). Also, let $S^{n-1}$ be the unit sphere in ${\mathbb R}^n$. 
{Similar arguments to those for \cite[Lemma 2.1, Remark 2.4]{Amrouche-3} imply that
any function $u$ in ${\mathcal H}^1_{p}({\mathbb R}^n)$ or ${\mathcal H}^1_{p}(\Omega _{-})$, with $1<p<n$, vanishes at infinity in the} sense of Leray, i.e.,
\begin{align}
{\lim _{r\to \infty }\int _{S^{n-1}}|u(r{\bf y})|d\sigma _{\bf y}=0.}
\end{align}
\end{remark}

\subsection{The conormal derivative operator for the $L_{\infty}$ coefficient Stokes system}

Recall that $\check{\boldsymbol{\mathcal L}}$ is a second-order elliptic differential operator in divergence form given by \eqref{Stokes-0}, where the coefficients $A^{\alpha \beta }$ of
${\mathbb A}=\left(A^{\alpha \beta }\right)_{1\leq \alpha ,\beta \leq n}$ are $n\times n$ matrix-valued functions on ${\mathbb R}^n$ with bounded measurable, real-valued entries $a_{ij}^{\alpha \beta }$, i.e.,
$A^{\alpha \beta }=\left\{a_{ij}^{\alpha \beta }\right\}_{1\leq i,j\leq n}$,
and  the {\it strong ellipticity condition} \eqref{mu} is satisfied.
Similar to \cite{Cai-Wang-2010, Camano-Gatica2016} and references therein, we can define the {\em non-symmetric pseudostress tensor} $\boldsymbol{\sigma}({\bf u},\pi)$ with components $\sigma_{\alpha i}(\mathbf u,\pi)=a_{ij}^{\alpha \beta }\partial_\beta u_j-\delta_{\alpha i}\pi$.

Let $\boldsymbol\nu =(\nu _1,\ldots ,\nu _n)^\top$ be the outward unit normal to $\Omega_{ +}$, which is defined a.e. on $\partial {\Omega }$.
When $({\bf u},\pi)\!\in C^1(\overline\Omega_{\pm})^n\times \!C^0(\overline\Omega_{\pm})$, 
the {\em classical} interior and exterior conormal derivatives (i.e., the {\it boundary pseudotractions}) for the Stokes operator
$
\boldsymbol{\mathcal L}({\bf u},\pi )=\partial_\alpha\left(A^{\alpha \beta }\partial_\beta {\bf u}\right)
-\nabla \pi
$
are
\begin{align*}
{\bf T}^{c\pm} ({\bf u},\pi):=\gamma_\pm\boldsymbol{\sigma}({\bf u},\pi)\cdot\boldsymbol\nu
=\gamma_\pm (A^{\alpha\beta}\partial_\beta{\bf u})\nu_\alpha -\gamma_\pm\pi\,\boldsymbol\nu\quad \mbox{on }\partial\Omega,
\end{align*}
cf., e.g., \cite{Choi-Dong-Kim-JMFM}.
Here and in the sequel, the indices $\pm $ mark the trace and conormal derivatives from $\Omega_\pm$, respectively.
Moreover, {the following first Green identity holds,}
\begin{align}
\label{special-case-1-var}
\!\!\!\!\!{{\pm}\left\langle {\bf T}^{c\pm}({\bf u},\pi ),
\boldsymbol\varphi \right\rangle _{_{\!\partial\Omega  }}}
&{=\big\langle A^{\alpha \beta }\partial _\beta {\bf u},\partial _\alpha \boldsymbol\varphi \big\rangle _{\Omega_{\pm}}
-\langle \pi,{\rm{div}}\, \boldsymbol\varphi \rangle _{\Omega_{\pm}}
+\left\langle \boldsymbol{\mathcal L}({\bf u},\pi ),\boldsymbol\varphi\right\rangle _{{\Omega_{\pm}}}},\ \forall \ \boldsymbol\varphi \in {\mathcal D}({\mathbb R}^n)^n.
\end{align}

\begin{definition}
\label{conormal-derivative-var-Brinkman}
For $p\in (1,\infty )$, let us define the space
\begin{align*}
{\boldsymbol{\mathcal H}}^1_{p}({\Omega_{\pm}},\boldsymbol{\mathcal L})
:=\Big\{&({\bf u}_{\pm},\pi_{\pm},{\tilde{\bf f}}_{\pm})\in
{\mathcal H}^1_{p}({\Omega_{\pm}})^n\times L_p({\Omega_{\pm}})\times \widetilde{\mathcal H}^{-1}_{p}({\Omega_{\pm}})^n:
{\boldsymbol{\mathcal L}}({\bf u}_{\pm},\pi_{\pm})={\tilde{\bf f}}_{\pm}|_{\Omega_{\pm}}
\mbox{ in } {\Omega_{\pm}}\Big\}.
\end{align*}
\end{definition}
{Formula \eqref{special-case-1-var}} suggests the weak definition of the {formal and} generalized conormal derivatives for the $L_{\infty}$ coefficient Stokes system in the setting of $L_p$-based weighted Sobolev spaces (cf., e.g., \cite[Lemma 3.2]{Co}, \cite[Lemma 2.9]{K-L-M-W},
\cite[Definition 3.1, Theorem 3.2]{Mikh}, \cite[Theorem 10.4.1]{M-W}).
\begin{definition}
\label{conormal-derivative-var-Brinkman}
Let $p\in (1,\infty )$.
For any $({\bf u}_{\pm},\pi_{\pm},{\tilde{\bf f}}_{\pm})\in
{\mathcal H}^1_{p}({\Omega_{\pm}})^n\times L_p({\Omega_{\pm}})\times
\widetilde{\mathcal H}^{-1}_{p}({\Omega_{\pm}})^n$,
the {\em {formal} conormal derivatives}
${\bf T}^{\pm}({\bf u}_{\pm},\pi_{\pm} ;{\tilde{\bf f}}_{\pm})\!\in \! B_{p,p}^{-\frac{1}{p}}(\partial\Omega )^n$
{are defined} as
\begin{multline}
\label{conormal-derivative-var-Brinkman-3}
\!\!\!\!\!{\pm}\big\langle {\bf T}^{\pm}({\bf u}_{\pm},\pi_{\pm} ;{\tilde{\bf f}}_{\pm}),\boldsymbol\Phi\big\rangle _{_{\!\partial\Omega  }}\!\!
:=\!\big\langle A^{\alpha \beta }\partial _\beta ({\bf u}_{\pm}),\partial _\alpha (\gamma^{-1}_{\pm}\boldsymbol\Phi)\big\rangle _{\Omega_{\pm}}\!
-\!\big\langle {\pi_{\pm}},{\rm{div}}(\gamma^{-1}_{\pm}\boldsymbol\Phi)\big\rangle _{\Omega_{\pm}}\!\\
+\!\big\langle {\tilde{\bf f}}_{\pm},\gamma^{-1}_{\pm}\boldsymbol\Phi\big\rangle _{{\Omega_{\pm}}},
\quad \forall\,\boldsymbol\Phi\!\in\!B_{p',p'}^{\frac{1}{p}}(\partial\Omega )^n,
\end{multline}
where $\gamma^{-1}_{\pm}\!:\!B_{p',p'}^{\frac{1}{p}}(\partial\Omega )^n\!\to \!{\mathcal H}^1_{p'}({\Omega_{\pm}})^n$ is a bounded right inverse of the trace operator $\gamma_{\pm}\!:\!{\mathcal H}^1_{p'}({\Omega_{\pm}})^n\!\to \!B_{p',p'}^{\frac{1}{p}}(\partial\Omega )^n$.

Moreover, if $({\bf u}_{\pm},\pi_{\pm} ,{\tilde{\bf f}}_{\pm})\in
{\boldsymbol{\mathcal H}}^1_{p}({\Omega_{\pm}},{\boldsymbol{\mathcal L}})$, equation
\eqref{conormal-derivative-var-Brinkman-3} defines
the {\em generalized conormal derivatives} ${\bf T}^{\pm}({\bf u}_{\pm},\pi_{\pm} ;{\tilde{\bf f}}_{\pm}) \in B_{p,p}^{-\frac{1}{p}}(\partial\Omega )^n$.
\end{definition}
In addition, we have the following assertion (see also \cite{Co}, \cite[Theorem 5.3]{Mikh-3}, \cite[Lemma 2.9]{K-L-M-W}, \cite[Theorem 10.4.1]{M-W}).
\begin{lemma}
\label{lem-add1}
Let $p\in (1,\infty )$. 
\begin{itemize}
\item[$(i)$]
{The formal conormal derivative operator
${\bf T}^{\pm }:{\mathcal H}^1_{p}({\Omega_{\pm}})^n\!\times \!L_p({\Omega_{\pm}})\times \widetilde{\mathcal H}^{-1}_{p}({\Omega_{\pm}})^n\!\to \!B_{p,p}^{-\frac{1}{p}}(\partial \Omega )^n$
is linear and continuous.}
\item[$(ii)$]
The {generalized} conormal derivative operator ${\bf T}^{\pm }:\boldsymbol{\mathcal H}^1_{p}(\Omega _{\pm },\boldsymbol{\mathcal L})\to B_{p,p}^{-\frac{1}{p}}(\partial \Omega )^n$ is linear and continuous, and definition \eqref{conormal-derivative-var-Brinkman-3} does not depend on the choice of a right inverse $\gamma _{\pm }^{-1}:B_{p',p'}^{\frac{1}{p}}(\partial \Omega )^n\to {\mathcal H}^1_{p'}(\Omega _\pm )^n$ of the trace operator $\gamma _{\pm }:{\mathcal H}^1_{p'}(\Omega )^n\to B_{p',p'}^{\frac{1}{p}}(\partial \Omega )^n$. 
In addition,
the first Green identity
\begin{align}
\label{Green-particular-p}
\!\!\!\!\!\!\!\!\!\!\!\!\!\!\!{\pm}\big\langle {\bf T}^{\pm}({\bf u}_{\pm},\pi_{\pm};{\tilde{\bf f}}_{\pm}),\gamma_{\pm}{\bf w}_\pm\big\rangle _{{\partial\Omega  }}\!=\!\big\langle A^{\alpha \beta }\partial _\beta ({\bf u}_{\pm}),\partial _\alpha ({\bf w}_\pm)\big\rangle _{\Omega_{\pm}}\!\!-\!\langle {\pi_{\pm}},{\rm{div}}\, {\bf w}_\pm \rangle _{\Omega_{\pm}}\!+\!\langle {\tilde{\bf f}}_{\pm},{\bf w}_\pm \rangle _{{\Omega_{\pm}}}\,,
\end{align}
holds for {any} ${\bf w}_\pm \!\in \!{\mathcal H}^1_{p'}({\Omega_{\pm}})^n$ and $({\bf u}_{\pm},\pi_{\pm},{\tilde{\bf f}}_{\pm})\in {\boldsymbol{\mathcal H}}^1_{p}({\Omega_{\pm}},{\boldsymbol{\mathcal L}})$.
\end{itemize}
\end{lemma}
The proof follows with similar arguments as those for \cite[Lemma 2.2]{K-L-W1} (see also
\cite[Definition 3.1, Theorem 3.2]{Mikh}, \cite{Mikh-3}). We omit the details for the sake of brevity.

For $({\bf u}_{\pm},\pi_{\pm},{\tilde{\bf f}}_{\pm})\in
{\mathcal H}^1_{p}({\Omega_{\pm}})^n\times L_p({\Omega_{\pm}})\times\widetilde{\mathcal H}^{-1}_{p}({\Omega_{\pm}})^n$, let us introduce the couples
${\bf u}:=\{{\bf u}_+,{\bf u}_-\},$
$\pi:=\{\pi_+, \pi_-\},$
$\tilde{\bf f}:=\{{\tilde{\bf f}}_+ , {\tilde{\bf f}}_-\},$
and denote the jump of the corresponding conormal derivatives by
\begin{align}
\label{jt}
&[{\bf T}({\bf u},\pi;\tilde{\bf f})]:=
\!{\bf T}^{+}({\bf u}_+,\pi_+;\tilde{\bf f}_+)\!-\!{\bf T}^{-}({\bf u}_-,\pi_-;\tilde{\bf f}_-).
\end{align}

For $({\bf u}_{\pm},\pi_{\pm})$ such that
$({\bf u}_{\pm},\pi_{\pm},{\bf 0})\in {\boldsymbol{\mathcal H}}^1_{p}({\Omega_{\pm}},{\boldsymbol{\mathcal L}})$,
we will also use the notations
${\bf T}^\pm({\bf u}^\pm,\pi^\pm):={\bf T}^\pm({\bf u}^\pm,\pi^\pm;{\bf 0})$
and
$[{\bf T}({\bf u},\pi)]:=[{\bf T}({\bf u},\pi;{\bf 0})].$

Lemma \ref{lem-add1} implies the following result.
\begin{lemma}
\label{lemma-add-new}
Let $p\in (1,\infty )$, $({\bf u}_{\pm},\pi_{\pm}$,
${\tilde{\bf f}}_{\pm})\in {\boldsymbol{\mathcal H}}^1_{p}({\Omega_{\pm}},{\boldsymbol{\mathcal L}})$,
and ${\bf w}\in {\mathcal H}^1_{p'}(\mathbb R^n)^n$.
Then
\begin{align}
\label{jump-conormal-derivative-1f}
\big\langle [{\bf T}({\bf u},\pi ;\tilde{\bf f})],\gamma{\bf w}\big\rangle _{_{\partial\Omega }}
=&
\big\langle A^{\alpha \beta }\partial _\beta ({\bf u}_+),\partial _\alpha ({\bf w})\big\rangle _{\Omega_+}
+\big\langle A^{\alpha \beta }\partial _\beta ({\bf u}_-),\partial_\alpha ({\bf w})\big\rangle _{\Omega_-}\nonumber\\
&
-\langle {\pi}_+,{\rm{div}}\, {\bf w}\rangle_{\Omega_+}
-\langle {\pi}_+,{\rm{div}}\, {\bf w}\rangle_{\Omega_-}
+\left\langle\tilde{\bf f}_+,{\bf w}\right\rangle_{\Omega_+}
+\left\langle\tilde{\bf f}_-,{\bf w}\right\rangle_{\Omega_-}.
\end{align}
Moreover, if
$({\bf u}_{\pm},\pi_{\pm},{\bf 0})\in {\boldsymbol{\mathcal H}}^1_{p}({\Omega_{\pm}},{\boldsymbol{\mathcal L}})$,
then
\begin{align}
\label{jump-conormal-derivative-1f0}
\big\langle [{\bf T}({\bf u},\pi)],\gamma{\bf w}\big\rangle _{_{\partial\Omega }}
=&
\big\langle A^{\alpha \beta }\partial _\beta ({\bf u}_+),\partial _\alpha ({\bf w})\big\rangle _{\Omega_+}
+\big\langle A^{\alpha \beta }\partial _\beta ({\bf u}_-),\partial_\alpha ({\bf w})\big\rangle _{\Omega_-}\nonumber\\
&
-\langle {\pi}_+,{\rm{div}}\, {\bf w}\rangle_{\Omega_+}
-\langle {\pi}_+,{\rm{div}}\, {\bf w}\rangle_{\Omega_-}.
\end{align}
\end{lemma}
\begin{proof}
It suffices to remark that $\gamma_+{\bf w}=\gamma_-{\bf w}=\gamma{\bf w}$ and apply formula \eqref{Green-particular-p}.
\end{proof}

\subsection{Conormal derivative for the adjoint system}
\label{adj-conormal}
The {formally} adjoint operator ${\boldsymbol{\mathcal L}}^*$ is defined by
\begin{align}
\label{Stokes-new-adjoint}
&\boldsymbol{\mathcal L}^*({\bf v},q):=
\partial _\alpha\big(A^{*\alpha\beta} \partial _\beta {\bf v}\big)-\nabla q,
\\ \nonumber
\mbox{where }\quad &\mathbb A^*
=\{A^{*\alpha\beta}\}_{1\leq \alpha ,\beta \leq n},\, A^{*\alpha \beta }
=\left\{a_{ij}^{*\alpha\beta}\right\}_{1\leq i,j\leq n},\, a_{ij}^{*\alpha\beta} :=a^{\beta\alpha}_{ji}.
\end{align}
Note that our notation $A^{*\alpha\beta}$ coincides with the notation $(A^{\beta\alpha})^\top $ in \cite{Choi-Dong-Kim-JMFM}. Evidently, the coefficients of ${\boldsymbol{\mathcal L}}^*$ also satisfy conditions \eqref{mu} with the same constant $c$.

If $({\bf v},q)\!\in \!C^1(\overline \Omega_{\pm})^n\!\times \!C^0(\overline\Omega_{\pm})$, the classical conormal derivative operator ${\bf T}^{*c\pm}$ associated with $\boldsymbol{\mathcal L}^*$ is defined by
\begin{align*}
{\bf T}^{*c\pm}({\bf v},q):=\gamma_\pm \big(A^{*\alpha\beta} \partial_\beta {\bf v}\big)\nu_\alpha -{\gamma_\pm}q\,\boldsymbol\nu\quad\mbox{on }\partial\Omega.
\end{align*}
For more general functions ${\bf v}$ and $q$, we can introduce, similar to Definition~\ref{conormal-derivative-var-Brinkman}, the notion of formal and generalized conormal derivatives associated with $\boldsymbol{\mathcal L}^*$. 
\begin{definition}
\label{conormal-derivative-var-Brinkman*}
Let $p\in (1,\infty )$.
For any $({\bf v}_{\pm},q_{\pm},{\tilde{\bf g}}_{\pm})\!\in \!{\mathcal H}^1_{p}({\Omega_{\pm}})^n\!\times \! L_p({\Omega_{\pm}})\!\times \!\widetilde{\mathcal H}^{-1}_{p}({\Omega_{\pm}})^n$,
the {\em formal conormal derivatives}
${\bf T}^{*\pm}({\bf v}_{\pm},q_{\pm},{\tilde{\bf g}}_{\pm})\!\in \! B_{p,p}^{-\frac{1}{p}}(\partial\Omega )^n$
{are defined} as
\begin{multline}
\label{conormal-derivative-var-Brinkman-3-adj}
{\pm}\left\langle {\bf T}^{*\pm}({\bf v}_{\pm},q_{\pm} ;{\tilde{\bf g}}_{\pm}),\boldsymbol\Phi\right\rangle _{_{\!\partial\Omega  }}\!\!
:=\!\!{\big\langle A^{*\alpha \beta } \partial _\beta ({\bf v}_{\pm}),\partial_\alpha (\gamma^{-1}_{\pm}\boldsymbol\Phi)\big\rangle _{\Omega_{\pm}}}\\
-\left\langle {q_{\pm}},{\rm{div}}(\gamma^{-1}_{\pm}\boldsymbol\Phi)\right\rangle _{\Omega_{\pm}}
+\!\left\langle {\tilde{\bf g}}_{\pm},\gamma^{-1}_{\pm}\boldsymbol\Phi\right\rangle _{{\Omega_{\pm}}},\
 \forall \, \boldsymbol\Phi\!\in\!B_{p',p'}^{\frac{1}{p}}(\partial\Omega )^n.
\end{multline}
Moreover, if
$({\bf v}_{\pm},q_{\pm},{\tilde{\bf g}}_{\pm})\in
{\boldsymbol{\mathcal H}}^1_{p}({\Omega_{\pm}},{\boldsymbol{\mathcal L}}^*)$,
equation \eqref{conormal-derivative-var-Brinkman-3-adj} defines
the {\em generalized conormal derivatives}
${\bf T}^{*\pm}({\bf v}_{\pm},q_{\pm},{\tilde{\bf g}}_{\pm})
\in B_{p,p}^{-\frac{1}{p}}(\partial\Omega )^n$.
\end{definition}
\begin{lemma}
\label{lem-add1*}
Let $p\in (1,\infty )$. 
\begin{itemize}
\item[$(i)$]
The formal conormal derivative operator
${\bf T}^{*\pm}\!:\!{\mathcal H}^1_{p}({\Omega_{\pm}})^n\times L_p({\Omega_{\pm}})\!\times \!\widetilde{\mathcal H}^{-1}_{p}({\Omega_{\pm}})^n\!\to \! B_{p,p}^{-\frac{1}{p}}(\partial \Omega )^n$
is linear and continuous.
\item[$(ii)$]
The generalized conormal derivative operator ${\bf T}^{*\pm}:\boldsymbol{\mathcal H}^1_{p}(\Omega _{\pm },\boldsymbol{\mathcal L}^*)\to B_{p,p}^{-\frac{1}{p}}(\partial \Omega )^n$ is linear and continuous, and definition \eqref{conormal-derivative-var-Brinkman-3-adj} does not depend on the choice of a right inverse $\gamma _{\pm}^{-1}\!:\!B_{p',p'}^{\frac{1}{p}}(\partial \Omega )^n\!\to \!{\mathcal H}^1_{p'}(\Omega _\pm )^n$ of the trace operator $\gamma _{\pm }\!:\!{\mathcal H}^1_{p'}(\Omega )^n\!\to \! B_{p',p'}^{\frac{1}{p}}(\partial \Omega )^n$.
In addition, the following first Green identity holds for {any} ${\bf w}_\pm \!\in \!{\mathcal H}^1_{p'}({\Omega_{\pm}})^n$ and $({\bf v}_{\pm},q_{\pm},{\tilde{\bf g}}_{\pm})\in {\boldsymbol{\mathcal H}}^1_{p}({\Omega_{\pm}},\boldsymbol{\mathcal L}^*)$
\begin{align}
\label{Green-particular-p-adj}
\!\!\!\!\!\!\!\!\!\!\!\!\!\!\!{\pm}\big\langle {\bf T}^{*\pm}({\bf v}_{\pm},q_{\pm};{\tilde{\bf g}}_{\pm}),
\gamma_{\pm}{\bf w}_\pm\big\rangle _{{\partial\Omega }}
\!
&=\!\big\langle A^{*\alpha \beta }\partial _\beta({\bf v}_\pm),\partial_\alpha({\bf w}_{\pm})\big\rangle _{\Omega_{\pm}}\!\!
-\!\langle {q_{\pm}},{\rm{div}}\, {\bf w}_\pm \rangle _{\Omega_{\pm}}\!+\!\langle {\tilde{\bf g}}_{\pm},{\bf w}_\pm \rangle _{{\Omega_{\pm}}}\nonumber\\
&=\!\big\langle A^{\alpha \beta }\partial _\beta({\bf w}_{\pm}),\partial_\alpha({\bf v}_\pm)\big\rangle _{\Omega_{\pm}}\!\!
-\!\langle {q_{\pm}},{\rm{div}}\, {\bf w}_\pm \rangle _{\Omega_{\pm}}\!+\!\langle {\tilde{\bf g}}_{\pm},{\bf w}_\pm \rangle _{{\Omega_{\pm}}}.
\end{align}
\end{itemize}
\end{lemma}

Lemma \ref{lem-add1*} implies the following analogue of Lemma~\ref{lemma-add-new}.
\begin{lemma}
\label{jump-conormal-derivative-1}
Let $p\in (1,\infty )$, $({\bf v}_{\pm},q_{\pm},{\tilde{\bf g}}_{\pm})\in
{\boldsymbol{\mathcal H}}^1_{p}({\Omega_{\pm}},{\boldsymbol{\mathcal L}}^*)$,
and ${\bf w}\in {\mathcal H}^1_{p'}(\mathbb R^n)^n$. Let ${\bf v}$ and $q$ be the couples $\{{\bf v}_+,{\bf v}_{-}\}$ and $\{q_+,q_{-}\}$.
Then
\begin{align}
\label{jump-conormal-derivative-1f*}
\big\langle [{\bf T}^*({\bf v},q;{\bf g})],\gamma{\bf w}\big\rangle _{_{\partial\Omega }}
=&\big\langle A^{*\alpha \beta }\partial_\beta({\bf v}_+),\partial_\alpha({\bf w})\big\rangle _{\Omega_+}
+\big\langle A^{*\alpha \beta }\partial_\beta({\bf v}_-),\partial_\alpha({\bf w})\big\rangle _{\Omega_-}\nonumber\\
&
-\langle q_+,{\rm{div}}\, {\bf w}\rangle_{\Omega_+}
-\langle q_+,{\rm{div}}\, {\bf w}\rangle_{\Omega_-}
+\left\langle\tilde{\bf g}_+,{\bf w}\right\rangle_{\Omega_+}
+\left\langle\tilde{\bf g}_-,{\bf w}\right\rangle_{\Omega_-}\,.
\end{align}
Moreover, if
$({\bf v}_{\pm},q_{\pm},{\tilde{\bf 0}}_{\pm})\in
{\boldsymbol{\mathcal H}}^1_{p}({\Omega_{\pm}},{\boldsymbol{\mathcal L}}^*)$,
then
\begin{align}
\label{jump-conormal-derivative-1-adj}
\big\langle [{\bf T}^*({\bf v},q)],\gamma{\bf w}\big\rangle _{_{\partial\Omega }}
=&\big\langle A^{*\alpha \beta }\partial_\beta({\bf v}_+),\partial_\alpha({\bf w})\big\rangle _{\Omega_+}
+\big\langle A^{*\alpha \beta }\partial_\beta({\bf v}_-),\partial_\alpha({\bf w})\big\rangle _{\Omega_-}\nonumber\\
&
-\langle q_+,{\rm{div}}\, {\bf w}\rangle_{\Omega_+}
-\langle q_{-},{\rm{div}}\, {\bf w}\rangle_{\Omega_-}.
\end{align}
\end{lemma}

\subsection{Abstract mixed variational formulations and well-posedness results}
\label{B-B-theory}

The main role in our analysis is played by the following well-posedness result
from \cite{Babuska}, \cite[Theorem 1.1]{Brezzi},
(cf., also \cite[Theorem 2.34]{Ern-Gu}, \cite{Brezzi-Fortin} and \cite[\S 4]{Gi-Ra}).
\begin{theorem}
\label{B-B}
Let $X$ and ${\mathcal M}$ be two real Hilbert spaces. Let $a(\cdot ,\cdot):X\times X\to {\mathbb R}$ and $b(\cdot ,\cdot):X\times {\mathcal M}\to {\mathbb R}$ be bounded bilinear forms. Let $f\in X'$ and $g\in {\mathcal M}'$. Let $V$ be the subspace of $X$ defined by
\begin{align}
\label{V}
V:=\left\{v\in X: b(v,q)=0,\ \forall \, q\in {\mathcal M}\right\}.
\end{align}
Assume that $a(\cdot ,\cdot ):V\times V\to {\mathbb R}$ is coercive, which means that there exists a constant $c_a>0$ such that
\begin{align}
\label{coercive}
a(w,w)\geq c_a\|w\|_X^2,\ \ \forall \, w\in V,
\end{align}
and that $b(\cdot ,\cdot)\!:\!X\!\times \!{\mathcal M}\!\to \!{\mathbb R}$ satisfies
the condition
\begin{align}
\label{inf-sup-sm}
&\inf _{q\in {\mathcal M}\setminus \{0\}}\sup_{v\in X\setminus \{0\}}\frac{b(v,q)}{\|v\|_X\|q\|_{\mathcal M}}\geq \beta \,,
\end{align}
with some constant $\beta>0$. Then the mixed variational problem
\begin{equation}
\label{mixed-variational}
\left\{\begin{array}{ll}
a(u,v)+b(v,p)&=f(v), \ \ \forall \, v\in X,\\
b(u,q)&=g(q), \ \ \forall \, q\in {\mathcal M},
\end{array}
\right.
\end{equation}
with the unknown $(u,p)\in X\times {\mathcal M}$, is well-posed, which means that \eqref{mixed-variational} has a unique solution $(u,p)$ in $X\times {\mathcal M}$ and there exists a constant $C>0$ depending on $\beta$ and  $c_a$, such that 
\begin{align}
\label{mixed-C}
\|u\|_{X}+\|p\|_{{\mathcal M}}\leq C\left(\|f\|_{X'}+\|g\|_{{\mathcal M}'}\right).
\end{align}
\end{theorem}
We will also need the following result (see \cite[Theorem A.56, Remark 2.7]{Ern-Gu}). 
\begin{lemma}
\label{surj-inj-inf-sup}
Let $X$ and ${\mathcal M}$ be reflexive Banach spaces. Let $b(\cdot ,\cdot):X\times {\mathcal M}\to {\mathbb R}$ be a bounded bilinear form. 
Let $B:X\to {\mathcal M}'$ and $B^{*}:{\mathcal M}\to X'$ be the linear bounded operators given by
\begin{align}
\label{B}
&\langle Bv,q\rangle =b(v,q),\ \langle v,B^*q\rangle =\langle Bv,q\rangle ,\ \forall \, v\in X,\ \forall \, q\in {\mathcal M},
\end{align}
where $\langle \cdot ,\cdot \rangle :=\!_{X'}\langle \cdot ,\cdot \rangle _X$ denotes the duality pairing of the dual spaces $X'$ and $X$. The duality pairing between ${\mathcal M}'$ and ${\mathcal M}$ is also denoted by $\langle \cdot ,\cdot \rangle $. Then the following assertions are equivalent:
\begin{itemize}
\item[$(i)$]
There exists a constant $\beta >0$ such that $b(\cdot ,\cdot)$ satisfies the inf-sup condition \eqref{inf-sup-sm}.
\item[$(ii)$]
The map $B:{X/V}\to {\mathcal M}'$ is an isomorphism and
$\|Bw\|_{{\mathcal M}'}\geq \beta \|w\|_{X/V},$ for any $w\in X/V.$
\end{itemize}
\end{lemma}

\section{Volume and layer potential operators for the $L_{\infty }$ coefficient Stokes system in $L_p$-based Sobolev and Besov spaces}\label{N-S-D}
In the sequel, $\Omega _+\subset \!{\mathbb R}^n$ ($n\geq 3)$ is a bounded Lipschitz domain with connected boundary $\partial \Omega $, and $\Omega _{-}\!:=\!{\mathbb R}^n\setminus \overline{\Omega }$. 
\subsection{Weak solution of the Stokes system with $L_{\infty }$ coefficients in ${\mathbb R}^n$.}
The main role in our analysis is played by the following result (see also \cite[Lemma 4.1]{K-M-W-1} for $p=2$). 
\begin{lemma}
\label{lemma-a47-1-Stokes}
Let $\mathbb A$ satisfy conditions \eqref{Stokes-1} and \eqref{mu}.
Let
$p\in (1,\infty )$,
and
$a_{\mathbb R^n}:{\mathcal H}^1_{p}({\mathbb R}^n)^n\times {\mathcal H}^1_{p'}({\mathbb R}^n)^n\!\to \!{\mathbb R}$,
$b_{\mathbb R^n}:{\mathcal H}^1_{p}({\mathbb R}^n)^n\times L_{p'}({\mathbb R}^n)\!\to \!{\mathbb R}$ be the bilinear forms
\begin{align}
\label{a-v}
&a_{{\mathbb R}^n}({\bf u},{\bf v}):=\big\langle A^{\alpha \beta }\partial _\beta {\bf u},
\partial _\alpha {\bf u}\big\rangle _{{\mathbb R}^n},\ \forall \, {\bf u}\in {\mathcal H}^1_{p}({\mathbb R}^n)^n,\, {\bf v}\in {\mathcal H}^1_{p'}({\mathbb R}^n)^n\,,\\
\label{b-v}
&b_{\mathbb R^n}({\bf v},q):=-\langle {\rm{div}}\, {\bf v},q\rangle _{{\mathbb R}^n},\ \forall \, {\bf v}\in {\mathcal H}^1_{p}({\mathbb R}^n)^n,\ \forall \, q\in L_{p'}({\mathbb R}^n)\,.
\end{align}
Then there exists $p_*\in (2,\infty )$ such that for any $p\in\mathcal R(p_*,n)$, where
\begin{align}
\label{p*}
\mathcal R(p_*,n):=\left(\frac{p_*}{p_*-1},p_*\right)\cap \left(\frac{n}{n-1},{n}\right)
\end{align}
and for all given data $\boldsymbol\xi \in {\mathcal H}^{-1}_{p}({\mathbb R}^n)^n$ and $\zeta \in L_p({\mathbb R}^n)$, the mixed variational formulation
\begin{align}
\label{transmission-S-variational-dl-3-equiv-0}
\left\{\begin{array}{ll}
a_{{\mathbb R}^n}({\bf u},{\bf v})+b_{\mathbb R^n}({\bf v},\pi )=\langle \boldsymbol\xi ,{\bf v}\rangle _{{\mathbb R}^n},\ \forall \, {\bf v}\in {\mathcal H}^1_{p'}({\mathbb R}^n)^n,\\
b_{\mathbb R^n}({\bf u},q)=\langle \zeta ,q\rangle _{{\mathbb R}^n},\ \forall \, q\in L_{p'}({\mathbb R}^n)
\end{array}
\right.
\end{align}
is well-posed, which means that \eqref{transmission-S-variational-dl-3-equiv-0} has a unique solution $({\bf u},\pi )\in {{\mathcal H}^1_{p}({\mathbb R}^n)^n}\times L_p({\mathbb R}^n)$ and there exists a constant $C=C(c_{\mathbb A},p,n)>0$ such that
\begin{align}
\label{estimate-1-wp-S}
\|{\bf u}\|_{{\mathcal H}^1_{p}({\mathbb R}^n)^n}+\|\pi \|_{L_p({\mathbb R}^n)}\leq
C\left\{\|\boldsymbol\xi \|_{{\mathcal H}^{-1}_{p}({\mathbb R}^n)^n}+\|\zeta \|_{L_p({\mathbb R}^n)}\right\}.
\end{align}
\end{lemma}
\begin{proof}
Inequalities \eqref{mu} combined with the H\"{o}lder inequality imply that there exists a constant $C=C(p, n,c_{\mathbb A})>0$ such that
\begin{align}
\label{a-1-v}
|a_{\mathbb R^n}({\bf u},{\bf v})|
\leq C\|{\bf u}\|_{{\mathcal H}^1_{p}({\mathbb R}^n)^n}\|{\bf v}\|_{{\mathcal H}^1_{p'}({\mathbb R}^n)^n},\ \forall\, {\bf u}\in {\mathcal H}^1_{p}({\mathbb R}^n)^n,\, {\bf v}\in \mathcal H^1_{p'}({\mathbb R}^n)^n.
\end{align}
Thus, the bilinear form $a_{{\mathbb R}^n}:{\mathcal H}^1_{p}({\mathbb R}^n)^n\times {\mathcal H}^1_{p'}({\mathbb R}^n)^n\to {\mathbb R}$ is bounded for any $p\in (1,\infty )$. 
The bilinear form $b_{\mathbb R^n}:{\mathcal H}^1_{p}({\mathbb R}^n)^n\!\times \!L_{p'}({\mathbb R}^n)\!\to \!{\mathbb R}$ is also bounded for any $p\!\in \!(1,\infty )$.

Let us first prove the lemma for $p=2$.
To do so, we intend to use Theorem \ref{B-B}, which requires the coercivity of the bilinear form $a_{{\mathbb R}^n}(\cdot ,\cdot )$ from ${\mathcal H}^{1}({\mathbb R}^n)^n\times {\mathcal H}^{1}({\mathbb R}^n)^n$ to ${\mathbb R}$.
Indeed, the strong ellipticity condition \eqref{mu} and the property that the semi-norm
is a norm on $\mathcal H^1({\mathbb R}^n)^n$ equivalent to the norm $\|\cdot \|_{\mathcal H^1({\mathbb R}^n)^n}$ (see \eqref{weight-2p} and \eqref{seminorm-R3} with $p=2$),
imply that there exists a constant $c_1=c_1(n)>0$ such that
\begin{align}
\label{a-1-v2-S}
a_{{\mathbb R}^n}({\bf v},{\bf v})\geq 
c_{\mathbb A}^{-1}\|\nabla {\bf v}\|_{L_2({\mathbb R}^n)^{n\times n}}^2
\geq 
c_{\mathbb A}^{-1}c_1\|{\bf v}\|_{{\mathcal H}^1({\mathbb R}^n)^n}^2,\, \forall \, {\bf v}\in {\mathcal H}^1({\mathbb R}^n)^n.
\end{align}
{Inequalities} \eqref{a-1-v} and \eqref{a-1-v2-S} show that the bilinear form $a_{{\mathbb R}^n}:{\mathcal H}^1({\mathbb R}^n)^n\times {\mathcal H}^1({\mathbb R}^n)^n\to {\mathbb R}$ is bounded and coercive.

Moreover, the boundedness of the operator
${\rm{div}}:{\mathcal H}^1({\mathbb R}^n)^n\to L_2({\mathbb R}^n)$
implies that the bilinear form 
$b_{\mathbb R^n}:{\mathcal H}^1({\mathbb R}^n)^n\times L_2({\mathbb R}^n)\to {\mathbb R}$ is bounded as well.
In addition, the subspace ${\mathcal H}^1_{\rm div}(\mathbb R^n)^n$ of ${\mathcal H}^1(\mathbb R^n)^n$-divergence free vector fields has the following characterization 
\begin{align}
{\mathcal H}^1_{\rm div}(\mathbb R^n)^n
=&\left\{{\bf w}\in {\mathcal H}^1({\mathbb R}^n)^n: b_{\mathbb R^n}({\bf w},q)\!=\!0,\ \forall \, q\in L_2({\mathbb R}^n) \right\}.\nonumber
\end{align}
In view of the isomorphism property of the operator
\begin{align}
\label{div-S}
-{\rm{div}}:{\mathcal H}^1({\mathbb R}^n)^n/{\mathcal H}^1_{\rm div}(\mathbb R^n)^n\to L_2(\mathbb R^n)
\end{align}
(cf. \cite[Proposition 2.1]{Al-Am-1}, \cite[Lemma 2.5]{Kozono-Shor}),
there exists a constant $c_2>0$ such that for any $q\in L_2({\mathbb R}^n)$ there exists ${\bf v}\in {\mathcal H}^1({\mathbb R}^n)^n$ satisfying the equation $-{\rm{div}}\, {\bf v}=q$ and the inequality
$\|\mathbf v\|_{{\mathcal H}^1({\mathbb R}^n)^n}\leq c_2\|q\|_{L_2({{\mathbb R}^n})}$, and hence
\begin{align*}
b_{\mathbb R^n}({\bf v},q)=-\left\langle {\rm{div}}\, {\bf v},q\right\rangle _{{\mathbb R}^n}
=\langle q,q\rangle _{{\mathbb R}^n}=\|q\|_{L_2({\mathbb R}^n)}^2
\geq c_2^{-1}\|{\bf v}\|_{{\mathcal H}^1({\mathbb R}^n)^n}\|q\|_{L_2({\mathbb R}^n)}.
\end{align*}
Consequently, the bilinear form $b_{\mathbb R^n}(\cdot ,\cdot):{\mathcal H}^1({\mathbb R}^n)^n\times L_2({\mathbb R}^n)\to {\mathbb R}$ satisfies the inf-sup condition
\begin{align*}
\inf _{q\in L_2({\mathbb R}^n)\setminus \{0\}}\sup _{{\bf w}\in {\mathcal H}^1({\mathbb R}^n)^n\setminus \{\bf 0\}}
\frac{b_{\mathbb R^n}({\bf w},q)}{\|{\bf w}\|_{{\mathcal H}^1({\mathbb R}^n)^n}\|q\|_{L_2({\mathbb R}^n)}}
\geq \inf _{q\in L_2({\mathbb R}^n)\setminus \{0\}}
\frac{c_2^{-1}\|{\bf v}\|_{{\mathcal H}^1({\mathbb R}^n)^n}\|q\|_{L_2({\mathbb R}^n)}}
{\|{\bf v}\|_{{\mathcal H}^1({\mathbb R}^n)^n}\|q\|_{L_2({\mathbb R}^n)}}
=c_2^{-1}
\end{align*}
(see also Lemma \ref{surj-inj-inf-sup}(ii), and \cite[Proposition 2.4]{Sa-Se} for $n=2,3$).
Then Theorem \ref{B-B}, with $X\!=\!{\mathcal H}^1({\mathbb R}^n)^n$, $M\!=\!L_2({\mathbb R}^n)$, $V\!\!={\mathcal H}^1_{\rm{div}}(\mathbb R^n)^n$, implies
that problem \eqref{transmission-S-variational-dl-3-equiv-0} is well posed for $p=2$.

Let
\begin{align}
\label{X-p-M}
{\mathcal X}_p({\mathbb R}^n):={\mathcal H}^1_{p}({\mathbb R}^n)^n\times L_{p}({\mathbb R}^n),\
{\mathcal X}'_{p'}({\mathbb R}^n):={\mathcal H}^{-1}_{p'}({\mathbb R}^n)^n\times L_{p'}({\mathbb R}^n).
\end{align}
and note that ${\mathcal X}'_{p'}({\mathbb R}^n)$ is the dual of the space ${\mathcal X}_{p}({\mathbb R}^n)$.
Let
$
\mathfrak{T}_{{\mathbb R}^n}=(\mathfrak{T}_{1;{\mathbb R}^n},\mathfrak{T}_{2;{\mathbb R}^n})
:{\mathcal X}_{p}({\mathbb R}^n)\to {\mathcal X}'_{p'}({\mathbb R}^n)
$
be the operator defined  on any  $({\bf u},\pi )\in {\mathcal X}_{p}({\mathbb R}^n)$ in the weak form by
\begin{align*}
\langle \mathfrak{T}_{1;{\mathbb R}^n}({\bf u},\pi ),\mathbf v\rangle _{{\mathbb R}^n}
=a_{{\mathbb R}^n}({\bf u},\mathbf v)+b_{\mathbb R^n}(\mathbf v,\pi ),\,
\langle \mathfrak{T}_{2;{\mathbb R}^n}({\bf u},\pi ),q\rangle _{{\mathbb R}^n}=b_{\mathbb R^n}({\bf u},q),
\ \forall \ (\mathbf v,q)\in {\mathcal X}_{p'}({\mathbb R}^n).
\end{align*}
Hence, establishing the existence of a solution to the variational problem \eqref{transmission-S-variational-dl-3-equiv-0} is equivalent to showing that the operator $\mathfrak{T}_{{\mathbb R}^n}:{\mathcal X}_{p}({\mathbb R}^n)\to {\mathcal X}'_{p'}({\mathbb R}^n)$ is an isomorphism
(see also \cite[Proposition 7.2]{B-M-M-M}, \cite[Theorem 5.6]{H-J-K-R}, and \cite[Theorem 3.1]{Ott-Kim-Brown} for the standard Stokes system).

The linear operator $\mathfrak{T}_{{\mathbb R}^n}:{\mathcal X}_{p}({\mathbb R}^n)\to {\mathcal X}'_{p'}({\mathbb R}^n)$ is continuous for any $p\in (1,\infty )$ due to \eqref{a-1-v}.
We already shown the operator $\mathfrak{T}_{{\mathbb R}^n}:{\mathcal X}_{p}({\mathbb R}^n)\to {\mathcal X}'_{p'}({\mathbb R}^n)$ is an isomorphism for $p=2$.
To show that it is also a isomorphism for $p$ in an open interval containing $2$, we proceed as follows.

Let us note that the sets $\{{\mathcal X}_{p}({\mathbb R}^n)\}_{p\in {\mathcal I}}$ and $\{{\mathcal X}_{p'}'({\mathbb R}^n)\}_{p\in {\mathcal I}}$ are both complex interpolation scales whenever {${\mathcal I}=(\frac{n}{n-1},n)$}.
To show this, we note that the sets $\{{\mathcal H}^1_{p}({\mathbb R}^n)\}_{1<p<n}$ and $\{L_p({\mathbb R}^n)\}_{p\in (1,\infty )}$ are complex interpolation scales (see \eqref{int-weight}, \cite[Theorem 3]{Triebel-2}, \cite[Theorem 2.4.2]{M-W}). Moreover, duality theorems for the complex method of interpolation imply that the dual of an interpolation scale is an interpolation scale itself (cf., e.g., \cite[Theorem 3.7.1, Corollary 4.5.2]{B-L}, \cite[p. 4391]{B-M-M-M}). Thus, starting with the complex interpolation scale {$\{{\mathcal H}^1_{p'}({\mathbb R}^n)\}_{1<p'<n}$}, we deduce by duality that the set {$\{{\mathcal H}^{-1}_{p}({\mathbb R}^n)\}_{p\in (\frac{n}{n-1},\infty )}$} is a complex interpolation scale as well.
Therefore, the range ${\mathcal I}$ of $p$ for which both sets {$\{{\mathcal H}^1_{p}({\mathbb R}^n)\}_{p\in {\mathcal I}}$} and {$\{{\mathcal H}^{-1}_{p}({\mathbb R}^n)\}_{p\in {\mathcal I}}$} are complex interpolation scales is the interval $(\frac{n}{n-1},n)$.
Consequently, the sets {$\{{\mathcal X}_{p}({\mathbb R}^n)\}_{\frac{n}{n-1}<p<n}$ and $\{{\mathcal X}_{p'}'({\mathbb R}^n)\}_{\frac{n}{n-1}<p<n}$} are complex interpolation scales.

Then the continuity of the operators $\mathfrak{T}_{{\mathbb R}^n}:{\mathcal X}_{p}({\mathbb R}^n)\to {\mathcal X}'_{p'}({\mathbb R}^n)$ for all $p\in (1,\infty )$, the isomorphism property of the operator $\mathfrak{T}_{{\mathbb R}^n}:{\mathcal X}_{2}({\mathbb R}^n)\to {\mathcal X}'_{2}({\mathbb R}^n)$, and 
the stability of the isomorphism property {on complex interpolation scales} (cf., e.g., \cite[Proposition 4.1]{Vignati}, \cite[Theorem 11.9.24]{M-W},
imply that there exists $p_*\in (2,\infty )$ such that for any $p\in \left(\frac{p_*}{p_*-1},p_*\right)\cap \left(\frac{n}{n-1},n\right)$ the operator $\mathfrak{T}_{{\mathbb R}^n}:{\mathcal X}_{p}({\mathbb R}^n)\!\to \!{\mathcal X}'_{p'}({\mathbb R}^n)$ 
is an isomorphism (see also \cite[Theorem 7.3]{B-M-M-M}, \cite[Theorem 5.6]{H-J-K-R}, \cite[Theorem 3.1]{Ott-Kim-Brown}).

Consequently, whenever condition \eqref{p*} holds 
and for all given data $(\boldsymbol\xi ,\zeta )\in {\mathcal H}^{-1}_{p}({\mathbb R}^n)^n\times L_p({\mathbb R}^n)$, there exists a unique solution $({\bf u},\pi )\!\in \!{{\mathcal H}^1_{p}({\mathbb R}^n)^n}\times L_p({\mathbb R}^n)$ of the equation $\mathfrak{T}_{{\mathbb R}^n}({\bf u},\pi )\!=\!(\boldsymbol\xi ,\zeta )$ or, equivalently, of the variational problem \eqref{transmission-S-variational-dl-3-equiv-0}, satisfying inequality \eqref{estimate-1-wp-S}.
\end{proof}

Next we use Lemma \ref{lemma-a47-1-Stokes} and show the well-posedness of the $L_{\infty}$-coefficient Stokes system in the space ${\mathcal H}^1_{p}({\mathbb R}^n)^n\!\times \!L_p({\mathbb R}^n)$ for any $p\in\mathcal R(p_*,n)$
(cf. \cite[Theorem 4.2]{K-M-W-1} for $p=2$ with $\mathbb A(x)=\mu(x)\mathbb I$, \cite[Proposition 2.9]{Kozono-Shor} and \cite[Theorem 3]{Al-Am-1} for $p\in (1,n)$ in the constant-coefficient case).
\begin{theorem}
\label{Brinkman-problem-p}
Let $\mathbb A$ satisfy conditions \eqref{Stokes-1} and \eqref{mu}.
Then there exists $p_*\in (2,\infty )$, 
such that for any $p\in{\mathcal R}(p_*,n)$, cf. \eqref{p*},
and for each $\mathbf f \in {\mathcal H}^{-1}_{p}({\mathbb R}^n)^n$,
the $L_{\infty}$-coefficient Stokes system
\begin{eqnarray}
\label{Newtonian-S-p}
\left\{
\begin{array}{ll}
\partial _\alpha\left(A^{\alpha \beta }\partial _\beta {\bf u}\right)-\nabla \pi=\mathbf f  & \mbox{ in } {\mathbb R}^n\,, 
\\
{\rm{div}}\, {\bf u}=0 & \mbox{ in } {\mathbb R}^n\,,
\end{array}\right.
\end{eqnarray}
has a unique solution $({\bf u}_{\mathbf f },\pi_{\mathbf f })\!\in \!{\mathcal H}^1_{p}({\mathbb R}^n)^n\!\times \!L_p({\mathbb R}^n)$ and there is a constant $C\!=\!C(c_{\mathbb A},p,n)\!>\!0$ such that
$\|{\bf u}_{\mathbf f }\|_{{\mathcal H}^1_{p}({\mathbb R}^n)^n}+ \|\pi_{\mathbf f }\|_{L_p({\mathbb R}^n)}\leq C\|\mathbf f \|_{{\mathcal H}^{-1}_{p}({\mathbb R}^n)^n}\,.$
\end{theorem}
\begin{proof}
Let $p_*\in (2,\infty )$ be as in Lemma \ref{lemma-a47-1-Stokes} and $p\in {\mathcal R}(p_*,n)$.
Then the dense embedding of the space ${\mathcal D}({\mathbb R}^n)^n$ in ${\mathcal H}^1_{p'}({\mathbb R}^n)^n$ 
shows that system \eqref{Newtonian-S-p} has the equivalent variational form \eqref{transmission-S-variational-dl-3-equiv-0} (with $\zeta =0$, $\boldsymbol\xi =-\mathbf f $), and the well-posedness of system \eqref{Newtonian-S-p} follows from Lemma \ref{lemma-a47-1-Stokes}.
\end{proof}
Theorem \ref{Brinkman-problem-p} allows us to define the Newtonian potential operators and show their continuity.
\begin{definition}
\label{Newtonian-B-var-variable-1}
Let $\mathbb A$ satisfy conditions \eqref{Stokes-1} and \eqref{mu}.
Let  $p_*\in (2,\infty )$ be as in Lemma \ref{lemma-a47-1-Stokes} and
$p\in {\mathcal R}(p_*,n)$, cf. \eqref{p*}.
For $\mathbf f \in {\mathcal H}^{-1}_{p}({\mathbb R}^n)^n$, we define the {\it Newtonian velocity and pressure potentials} for the $L_{\infty}$-coefficient Stokes system, by setting
\begin{align*}
\boldsymbol{\mathcal N}_{\mathbb R^n}{\mathbf f }:={\bf u}_{\mathbf f },\
{\mathcal Q}_{\mathbb R^n}{\mathbf f }:=\pi _{\mathbf f },
\end{align*}
where $({\bf u}_{\mathbf f },\pi _{\mathbf f })\in {\mathcal H}^1_{p}({\mathbb R}^n)^n\times L_p({\mathbb R}^n)$ is the unique solution of problem \eqref{Newtonian-S-p} with the given datum $\mathbf f $.
\end{definition}
\begin{lemma}
\label{Newtonian-B-var-1}
Let $\mathbb A$ satisfy conditions \eqref{Stokes-1} and \eqref{mu}. 
Let  $p_*\in (2,\infty )$ be as in Lemma \ref{lemma-a47-1-Stokes} and $p\in {\mathcal R}(p_*,n)$, cf. \eqref{p*}.
Then the following operators are linear and continuous
\begin{align}
\label{Newtonian-S-var-2}
\boldsymbol{\mathcal N}_{\mathbb R^n}:{\mathcal H}^{-1}_{p}({\mathbb R}^n)^n\to {\mathcal H}^1_{p}({\mathbb R}^n)^n,\
{\mathcal Q}_{\mathbb R^n}:{\mathcal H}^{-1}_{p}({\mathbb R}^n)^n\to L_p({\mathbb R}^n).
\end{align}
\end{lemma}

\subsection{The single layer potential operator for the Stokes system with $L_{\infty }$ coefficients}
Next we show a well-posedness result for a transmission problem and use it to define $L_{\infty}$-coefficient Stokes single layer potentials in Besov spaces $B_{p,p}^{-\frac{1}{p}}(\partial \Omega )^n$ with $p$ as in Lemma \ref{lemma-a47-1-Stokes} (cf. also \cite[Propositions 5.1, 7.1]{Sa-Se}, \cite[Theorem 4.5]{K-L-W1} for $p=2$, \cite[Propositions 2.3, 2.7]{B-H} for $p=2$, for the Stokes and Brinkman systems with constant coefficients in ${\mathbb R}^n$, $n\in \{2,3\}$.)

Recall that in this paper we assume that $\Omega _{+}\subset {\mathbb R}^n$ ($n\geq 3$) is a bounded Lipschitz domain with connected boundary $\partial \Omega $, and  
$\Omega _{-}:={\mathbb R}^n\setminus \overline{\Omega}_+$.
\begin{theorem}
\label{slp-var-apr-1-p}
Let $\mathbb A$ satisfy conditions \eqref{Stokes-1} and \eqref{mu}, 
$p_*\in (2,\infty )$ be as in Lemma~\ref{lemma-a47-1-Stokes}
and $p\in {\mathcal R}(p_*,n)$, cf. \eqref{p*}.
Then for any 
${\boldsymbol\psi \in B_{p,p}^{-\frac{1}{p}}(\partial \Omega )^n}$, the transmission problem
\begin{eqnarray}
\label{var-Brinkman-transmission-sl}
\left\{
\begin{array}{ll}
\partial _\alpha\left(A^{\alpha \beta }\partial_\beta {\bf u}\right)-\nabla \pi ={\bf 0} & 
\mbox{ in } \mathbb R^n\setminus \partial\Omega ,
\\
{\rm{div}}\, {\bf u}=0 & \mbox{ in } \mathbb R^n\setminus \partial\Omega ,\
\\
\left[{\gamma }({\bf u})\right]={\bf 0},\
\left[{\bf T}({\bf u},\pi )\right]
=\boldsymbol\psi & \mbox{ on } \partial \Omega ,
\end{array}\right.
\end{eqnarray}
has a unique solution $({\bf u}_{{\boldsymbol\psi }},\pi _{\boldsymbol\psi })\in {\mathcal H}^1_{p}({\mathbb R}^n)^n\times L_{p}({\mathbb R}^n)$,
and there exists a constant $C=C(\partial \Omega ,c_{\mathbb A},p,n)>0$ such that 
\begin{align*}
\|{\bf u}_{\boldsymbol\psi}\|_{{\mathcal H}^1_{p}({\mathbb R}^n)^n}+\|\pi _{\boldsymbol\psi }\|_{L_{p}({\mathbb R}^n)}\leq C\|\boldsymbol\psi\|_{B_{p,p}^{-\frac{1}{p}}(\partial \Omega )^n}.
\end{align*}
\end{theorem}
\begin{proof}
First, we note that the last condition in \eqref{var-Brinkman-transmission-sl} is understood in the sense of distributions, as in Definition \ref{conormal-derivative-var-Brinkman}.
Next, we show that the transmission problem \eqref{var-Brinkman-transmission-sl} has the following equivalent mixed variational formulation:

{\it Find $({\bf u}_{\boldsymbol\psi},\pi _{{\boldsymbol\psi }})\in {\mathcal H}^1_{p}({\mathbb R}^n)^n\times L_{p}({\mathbb R}^n)$ such that
\begin{align}
\label{single-layer-S-transmission}
\left\{\begin{array}{ll}
a_{{\mathbb R}^n}({\bf u}_{\boldsymbol\psi },{\bf v})+b_{\mathbb R^n}({\bf v},\pi _{\boldsymbol\psi })=\langle \boldsymbol\psi ,\gamma {\bf v}\rangle _{\partial \Omega }, & \forall \, {\bf v}\in {\mathcal H}^1_{p'}({\mathbb R}^n)^n,\\
b_{\mathbb R^n}({\bf u}_{\boldsymbol\psi },q)=0, & \forall \, q\in L_{p'}({\mathbb R}^n),
\end{array}
\right.
\end{align}
where $a_{{\mathbb R}^n}$ and $b_{\mathbb R^n}$ are the bilinear forms given by \eqref{a-v} and \eqref{b-v}.
}

First, assume that the pair $({\bf u}_{{\boldsymbol\psi }},\pi _{\boldsymbol\psi })\in {\mathcal H}^1_{p}({\mathbb R}^n)^n\times {L_{p}({\mathbb R}^n)}$ satisfies the transmission problem \eqref{var-Brinkman-transmission-sl}.
Then formula \eqref{jump-conormal-derivative-1f} shows that the same pair satisfies also the first equation in \eqref{single-layer-S-transmission}. The second equation of the mixed variational formulation \eqref{single-layer-S-transmission} follows from the fact that ${\bf u}_{{\boldsymbol\psi }}\in {\mathcal H}^1_{p}({\mathbb R}^n)^n$ satisfies the second equation in \eqref{var-Brinkman-transmission-sl}.
Conversely, assume that the pair $({\bf u}_{{\boldsymbol\psi }},\pi _{\boldsymbol\psi })\in {\mathcal H}^1_{p}({\mathbb R}^n)^n\times L_{p}({\mathbb R}^n)$ is a solution of the mixed variational formulation \eqref{single-layer-S-transmission}. In view of the density of the space ${\mathcal D}({\mathbb R}^n)^n$ in ${\mathcal H}^1_{p'}({\mathbb R}^n)^n$, and by choosing in the first equation of the system \eqref{single-layer-S-transmission} any ${\bf v}\in C^{\infty }({\mathbb R}^n)^n$ with compact support in $\Omega _\pm $ (and, thus, $\gamma {\bf v}={\bf 0}$), we obtain the variational equation
$\big\langle\partial _\alpha\big(A^{\alpha \beta }\partial _\beta ({\bf u}_{\boldsymbol\psi })\big)-\nabla \pi _{\boldsymbol\psi },{\bf w}\big\rangle _{\Omega _\pm}=0,\, \forall \, {\bf w}\in C_0^{\infty }(\Omega _\pm )^n\,,$
which yields the first equation in \eqref{var-Brinkman-transmission-sl}. The second equation in \eqref{var-Brinkman-transmission-sl} follows immediately from the second equation in \eqref{single-layer-S-transmission}, the property that the operator ${\rm{div}}:{\mathcal H}^1_{p}({\mathbb R}^n)^n\to L_{p}({\mathbb R}^n)$ is surjective (cf. \cite[Proposition 2.1]{Al-Am-1}, see also \cite[Proposition 2.4]{Sa-Se} for $p=2$), and the duality between the spaces $L_p({\mathbb R}^n)$ and $L_{p'}({\mathbb R}^n)$.
The assumption ${\bf u}_{\boldsymbol\psi }\in {\mathcal H}^1_{p}({\mathbb R}^n)^n$ implies the first transmission condition in \eqref{var-Brinkman-transmission-sl}. Using again formula \eqref{jump-conormal-derivative-1f}, the first equation in \eqref{single-layer-S-transmission}, and Lemma \ref{trace-operator1},  we obtain the relation
$\left\langle [{\bf T}({\bf u}_{\boldsymbol\psi },\pi _{\boldsymbol\psi })]-\boldsymbol \psi ,\boldsymbol\Phi\right\rangle _{\partial \Omega }=0,$ for any $\boldsymbol\Phi\in B_{p',p'}^{\frac{1}{p}}(\partial \Omega )^n$
and hence the second transmission condition in \eqref{var-Brinkman-transmission-sl}.

In addition, the continuity of the trace operator $\gamma :{\mathcal H}^1_{p'}({\mathbb R}^n)^n\to B_{p',p'}^{\frac{1}{p}}(\partial \Omega )^n$ and of its adjoint $\gamma ^*:B_{p,p}^{-\frac{1}{p}}(\partial \Omega )^n\to {\mathcal H}^{-1}_{p}({\mathbb R}^n)^n$ implies the continuity of the linear form
\begin{align}
\boldsymbol\ell :{\mathcal H}^1_{p'}({\mathbb R}^n)^n\to {\mathbb R},\, \boldsymbol\ell({\bf v}):=\langle \boldsymbol\psi ,\gamma {\bf v}\rangle _{\partial \Omega }=\langle \gamma ^*\boldsymbol\psi ,{\bf v}\rangle _{{\mathbb R}^n},\ \forall \, {\bf v}\in {\mathcal H}^1_{p'}({\mathbb R}^n)^n\,.
\end{align}
According to Lemma \ref{lemma-a47-1-Stokes} there exists
$p_*\in (2,\infty )$, 
such that for any $p$ as in \eqref{p*}
and for any $\boldsymbol\psi \in B_{p,p}^{-\frac{1}{p}}({\mathbb R}^n)^n$, problem \eqref{single-layer-S-transmission} has a unique solution $({\bf u}_{{\boldsymbol\psi }},\pi _{\boldsymbol\psi })\in {\mathcal H}^1_{p}({\mathbb R}^n)^n\times L_p({\mathbb R}^n)$, which depends continuously on $\boldsymbol\psi $. Moreover, the equivalence between problems \eqref{var-Brinkman-transmission-sl} and \eqref{single-layer-S-transmission} shows that $({\bf u}_{\boldsymbol\psi },\pi _{\boldsymbol\psi })\!\in \!{\mathcal H}^1_{p}({\mathbb R}^n)^n\times L_p({\mathbb R}^n)$ is the unique solution of the transmission problem \eqref{var-Brinkman-transmission-sl}.
\end{proof}
The next result can be proved by the arguments similar to those in the proof of Theorem \ref{slp-var-apr-1-p}, mainly based on the Green formula \eqref{Green-particular-p-adj}.
\begin{theorem}
\label{slp-var-apr-1-p-adj}
Let $\mathbb A$ satisfy conditions \eqref{Stokes-1} and \eqref{mu}.
Then there exists $p_*\in (2,\infty )$, 
such that for any $p'\in{\mathcal R}(p_*,n)$, cf. \eqref{p*},
and for any 
${\boldsymbol\psi^* \in B_{p',p'}^{-\frac{1}{p'}}(\partial \Omega )^n}$, the transmission problem for the adjoint Stokes system
\begin{eqnarray}
\label{var-Brinkman-transmission-sl-adj}
\left\{
\begin{array}{ll}
\partial _\alpha\big(A^{*\alpha\beta}\partial_\beta {\bf v}\big)-\nabla q={\bf 0} & \mbox{ in } {\mathbb R}^n\setminus \partial \Omega ,
\\
{\rm{div}}\, {\bf v}=0 & \mbox{ in } {\mathbb R}^n\setminus \partial \Omega ,\
\\
\left[{\gamma }({\bf v})\right]={\bf 0},\
\left[{\bf T}^*({\bf v},q)\right]
=\boldsymbol\psi^* & \mbox{ on } \partial \Omega ,
\end{array}\right.
\end{eqnarray}
has a unique solution $({\bf v}_{{\boldsymbol\psi^*}},q_{\boldsymbol\psi^*})\in {\mathcal H}^1_{p'}({\mathbb R}^n)^n\times L_{p'}({\mathbb R}^n)$, and there exists $C_*\!=\!C_*(\partial \Omega ,c_{\mathbb A},p',n)\!>\!0$ such that
$\|{\bf v}_{\boldsymbol\psi^*}\|_{{\mathcal H}^1_{p'}({\mathbb R}^n)^n}+\|q_{\boldsymbol\psi^* }\|_{L_{p'}({\mathbb R}^n)}\leq C_*\|\boldsymbol\psi^*\|_{B_{p',p'}^{-\frac{1}{p'}}(\partial \Omega )^n}.$
\end{theorem}
Theorem \ref{slp-var-apr-1-p} plays a key role in the following definition (cf. \cite[p. 75]{Sa-Se} and \cite[Corollary 2.5]{B-H} for the isotropic constant-coefficient case and $p=2$, and \cite[formula (4.2), Lemma 4.6]{Barton} for strongly elliptic operators).
\begin{definition}
\label{s-l-S-variational-variable}
Let $\mathbb A$ satisfy conditions \eqref{Stokes-1} and \eqref{mu}, $p_*\in (2,\infty )$ be as in Lemma~\ref{lemma-a47-1-Stokes} and  $p\in {\mathcal R}(p_*,n)$, cf. \eqref{p*}.
Then for any $\boldsymbol\psi \in B_{p,p}^{-\frac{1}{p}}(\partial \Omega )^n$ we define the {\it single layer velocity and pressure potentials} with the density $\boldsymbol\psi$
 for the Stokes operator $\boldsymbol{\mathcal L}$ with coefficients 
${\mathbb A}$, as
\begin{align}
\label{slp-S-vp-var}
{\bf V}_{\partial\Omega}{\boldsymbol\psi}:={\bf u}_{{\boldsymbol\psi}},\ {\mathcal Q}^s_{\partial\Omega}{\boldsymbol\psi}:=\pi _{{\boldsymbol\psi}}\,,
\end{align}
and the boundary operators ${\boldsymbol{\mathcal V}}_{\partial\Omega}:B_{p,p}^{-\frac{1}{p}}(\partial \Omega )^n\to B_{p,p}^{1-\frac{1}{p}}(\partial \Omega )^n$ and
${{\mathcal K}_{\partial\Omega}}:B_{p,p}^{-\frac{1}{p}}(\partial \Omega )^n\to B_{p,p}^{-\frac{1}{p}}(\partial \Omega )^n$ as
\begin{align}
\label{slp-S-oper-var}
{\boldsymbol{\mathcal V}}_{\partial\Omega}{\boldsymbol\psi}:=\gamma{\bf u}_{\boldsymbol\psi},\
{\mathcal K}_{\partial\Omega}\boldsymbol\psi:=
\frac{1}{2}\left(
{\bf T}^+({\bf u}_{\boldsymbol\psi},\pi_{\boldsymbol\psi})
+{\bf T}^-({\bf u}_{\boldsymbol\psi},\pi_{\boldsymbol\psi})\right),
\end{align}
where $({\bf u}_{{\boldsymbol\psi}},\pi _{{\boldsymbol\psi}})$ is the unique solution of the transmission problem \eqref{var-Brinkman-transmission-sl} in ${\mathcal H}^1_{p}({\mathbb R}^n)^n\times L_p({\mathbb R}^n)$.
\end{definition}
The well-posedness of the transmission problem \eqref{var-Brinkman-transmission-sl} proved in Theorem \ref{slp-var-apr-1-p}, definitions \eqref{slp-S-oper-var} and the transmission conditions in \eqref{var-Brinkman-transmission-sl} imply the following assertion (cf. \cite[Propositions 5.2 and 5.3]{Sa-Se}, \cite[Lemma A.4, (A.10), (A.12)]{K-L-M-W} and \cite[Theorem 10.5.3]{M-W} for ${\mathbb A}=\mathbb I$).
\begin{lemma}
\label{continuity-sl-S-h-var}
Let $\mathbb A$ satisfy conditions \eqref{Stokes-1} and \eqref{mu}, $p_*\in (2,\infty )$ be as in Lemma~\ref{lemma-a47-1-Stokes} and  $p\in {\mathcal R}(p_*,n)$, cf. \eqref{p*}.
Then the following operators are linear and continuous
\begin{align}
\label{s-l-S-v1-var}
&{\bf V}_{\partial\Omega}:B_{p,p}^{-\frac{1}{p}}(\partial \Omega )^n\to {\mathcal H}^1_{p}({\mathbb R}^n)^n,\
{\mathcal Q}_{\partial\Omega}^s:B_{p,p}^{-\frac{1}{p}}(\partial \Omega )^n\to L_p({\mathbb R}^n),
\\
\label{s-l-S-v2-var}
&\boldsymbol{\mathcal V}_{\partial\Omega}:B_{p,p}^{-\frac{1}{p}}(\partial \Omega )^n\to B_{p,p}^{1-\frac{1}{p}}(\partial \Omega )^n,\, 
{\mathcal K}_{\partial\Omega}:B_{p,p}^{-\frac{1}{p}}(\partial \Omega )^n\to B_{p,p}^{-\frac{1}{p}}(\partial \Omega )^n.
\end{align}
For any $\boldsymbol\psi \in B_{p,p}^{-\frac{1}{p}}(\partial \Omega )^n$, the following jump relations hold a.e. on $\partial \Omega $
\begin{align}
&\gamma_\pm{\bf V}_{\partial\Omega}\boldsymbol\psi ={\boldsymbol{\mathcal V}}_{\partial\Omega}{\boldsymbol\psi}, 
\label{sl-identities-var-1}
\quad
{\bf T}^{\pm }\left({\bf V}_{\partial\Omega}{\boldsymbol\psi },{\mathcal Q}_{\partial\Omega}^s{\boldsymbol\psi }\right)
=\pm \frac{1}{2}\boldsymbol\psi
+{\mathcal K}_{\partial\Omega}\boldsymbol\psi .
\end{align}
\end{lemma}
By using Theorem \ref{slp-var-apr-1-p-adj} we can also define the {\it single layer potential operators, ${\bf V}^*_{\partial\Omega}$ and ${\mathcal Q}^{s*}_{\partial\Omega}$, of the adjoint Stokes system \eqref{var-Brinkman-transmission-sl-adj}}.
\begin{definition}
\label{s-l-S-variational-variable-adj}
Let $\mathbb A$ satisfy conditions \eqref{Stokes-1} and \eqref{mu}.
Let $p_*\in (2,\infty )$ be as in Theorem $\ref{slp-var-apr-1-p-adj}$ and $p'\in {\mathcal R}(p_*,n)$, cf. \eqref{p*}.
Then for any $\boldsymbol\psi^* \in B_{p',p'}^{-\frac{1}{p'}}(\partial \Omega )^n$, we define the {\it single layer velocity and pressure potentials} with the density $\boldsymbol\psi^* $ for the adjoint Stokes operator $\boldsymbol{\mathcal L}^*$ defined in \eqref{Stokes-new-adjoint}, with coefficients $\mathbb A$, by setting
\begin{align*}
{\bf V}^*_{\partial\Omega}{\boldsymbol\psi^* }:={\bf v}_{{\boldsymbol\psi^* }},\ 
{\mathcal Q}^{s*}_{\partial\Omega}{\boldsymbol\psi^*}:=\pi _{{\boldsymbol\psi^* }},
\end{align*}
and the {\it operators} ${\boldsymbol{\mathcal V}}^*_{\partial\Omega}:B_{p',p'}^{-\frac{1}{p'}}(\partial \Omega )^n\to B_{p',p'}^{1-\frac{1}{p'}}(\partial \Omega )^n$ and
${\mathcal K}^*_{\partial\Omega}:B_{p',p'}^{-\frac{1}{p'}}(\partial \Omega )^n\to B_{p',p'}^{-\frac{1}{p'}}(\partial \Omega )^n$ as
\begin{align}
\label{slp-S-oper-var-adjoint}
{\boldsymbol{\mathcal V}}^*_{\partial\Omega}{\boldsymbol\psi^* }:=\gamma{\bf v}_{\boldsymbol\psi^* },\
{\mathcal K}^*_{\partial\Omega}\boldsymbol\psi^*:=
\frac{1}{2}\left(
{\bf T}^{*+}({\bf v}_{\boldsymbol\psi^* },\pi_{\boldsymbol\psi^* })
+{\bf T}^{*-}({\bf v}_{\boldsymbol\psi^* },\pi_{\boldsymbol\psi^* })\right),
\end{align}
where $({\bf v}_{{\boldsymbol\psi^*}},\pi _{{\boldsymbol\psi^*}})$ is the unique solution of the transmission problem \eqref{var-Brinkman-transmission-sl-adj} in ${\mathcal H}^1_{p'}({\mathbb R}^n)^n\times L_{p'}({\mathbb R}^n)$.
\end{definition}

\begin{lemma}
\label{self-adj}
Let $\mathbb A$ satisfy conditions \eqref{Stokes-1} and \eqref{mu}.
Let $p_*\in (2,\infty )$ be as in Theorem $\ref{slp-var-apr-1-p}$ and $p\in {\mathcal R}(p_*,n)$, cf. \eqref{p*}, $\boldsymbol\psi  \in B_{p,p}^{-\frac{1}{p}}(\partial \Omega )^n$, 
$\boldsymbol\psi^* \in B_{p',p'}^{-\frac{1}{p'}}(\partial \Omega )^n$.
Then 
\begin{align}
\label{sl-identities-var-1-adjoint}
&\left[\gamma {\bf V}^*_{\partial\Omega}{\boldsymbol\psi^* }\right]={\bf 0},\
{\bf T}^{*\pm }\left({\bf V}^*_{\partial\Omega}{\boldsymbol\psi^* },{\mathcal Q}^{s*}_{\partial\Omega}{\boldsymbol\psi^* }\right)
=\pm \frac{1}{2}\boldsymbol\psi^*
+{\mathcal K}^*_{\partial\Omega}\boldsymbol\psi^* ,\\
\label{sa-sl}
&\left\langle \boldsymbol\psi ,{\mathcal V}^*_{\partial\Omega}\boldsymbol\psi^* \right\rangle _{\partial \Omega }
=\left\langle  {\mathcal V}_{\partial\Omega}\boldsymbol\psi, \boldsymbol\psi^* \right\rangle _{\partial \Omega }.
\end{align}
\end{lemma}
\begin{proof}
Formulas \eqref{sl-identities-var-1-adjoint} follow with arguments similar to those for \eqref{sl-identities-var-1}.
By definition, the couple $\big({\bf V}_{\partial\Omega}{\boldsymbol\psi },{\mathcal Q}_{\partial\Omega}^s{\boldsymbol\psi }\big)$ is the unique solution in ${\mathcal H}^1_{p}({\mathbb R}^n)^n\times L_p({\mathbb R}^n)$ of the transmission problem \eqref{var-Brinkman-transmission-sl} with the given datum $\boldsymbol\psi \in B_{p,p}^{-\frac{1}{p}}(\partial\Omega )^n$. 
Also $\big({\bf V}^*_{\partial\Omega}{\boldsymbol\psi^* },{\mathcal Q}^{s*}_{\partial\Omega}\boldsymbol\psi^* \big)$ is the unique solution in ${\mathcal H}^1_{p'}({\mathbb R}^n)^n\!\times \!L_{p'}({\mathbb R}^n)$ of the transmission problem for the adjoint Stokes system \eqref{var-Brinkman-transmission-sl-adj} with the given datum $\boldsymbol\psi^* \!\in \!B_{p',p'}^{-\frac{1}{p'}}(\partial \Omega )^n$. 
Then the Green formulas \eqref{jump-conormal-derivative-1f0} and \eqref{jump-conormal-derivative-1-adj} imply
\begin{align}
\label{jump-conormal-derivative-1-sl}
&\big\langle \big[{\bf T}\big({\bf V}_{\partial\Omega}{\boldsymbol\psi},{\mathcal Q}_{\partial\Omega}^s{\boldsymbol\psi }\big)\big],{\mathcal V}^*_{\partial\Omega}{\boldsymbol\psi^* }\big\rangle _{\partial \Omega }
=\big\langle A^{\alpha \beta }\partial _\beta \big({\bf V}_{\partial\Omega}{\boldsymbol\psi }\big),\partial _\alpha \big({\bf V}^*_{\partial\Omega}{\boldsymbol\psi^* }\big)\big\rangle _{{\mathbb R}^n}\\
\label{jump-conormal-derivative-1-sl-adj}
&\big\langle \big[{\bf T}^*\big({\bf V}^*_{\partial\Omega}{\boldsymbol\psi^* },{\mathcal Q}^{s*}_{\partial\Omega}\boldsymbol\psi^* \big)\big],{\mathcal V}_{\partial\Omega}{\boldsymbol\psi }\big\rangle _{\partial \Omega }
=\big\langle A^{*\alpha \beta }\partial_\beta\big({\bf V}^*_{\partial\Omega}{\boldsymbol\psi^*}\big),
\partial_\alpha \big({\bf V}_{\partial\Omega}{\boldsymbol\psi}\big)\big\rangle_{{\mathbb R}^n}
\nonumber\\
&\hspace{0.3em}=\big\langle a^{\alpha\beta}_{ij}\partial _\beta \big({\bf V}_{\partial\Omega}{\boldsymbol\psi }\big)_j,
\partial _\alpha\big({\bf V}^*_{\partial\Omega}{\boldsymbol\psi^* }\big)_i\big\rangle _{{\mathbb R}^n}
=\big\langle A^{\alpha \beta }\partial _\beta \big({\bf V}_{\partial\Omega}{\boldsymbol\psi }\big),\partial _\alpha \big({\bf V}^*_{\partial\Omega}{\boldsymbol\psi^* }\big)\big\rangle _{{\mathbb R}^n}.
\end{align}
Moreover, the second formulas in \eqref{sl-identities-var-1} and \eqref{sl-identities-var-1-adjoint} 
imply that
\begin{align}
\label{jsl}
\big[{\bf T}\big({\bf V}_{\partial\Omega}{\boldsymbol\psi },{\mathcal Q}_{\partial\Omega}^s{\boldsymbol\psi }\big)\big]=\boldsymbol\psi ,\
\left[{\bf T}^*\left({\bf V}^*_{\partial\Omega}{\boldsymbol\psi^* },{\mathcal Q}^{s*}_{\partial\Omega}\boldsymbol\psi^* \right)\right]=\boldsymbol\psi^* .
\end{align}
Then equality \eqref{sa-sl} follows from \eqref{jump-conormal-derivative-1-sl}, \eqref{jump-conormal-derivative-1-sl-adj} and \eqref{jsl} (see also \cite[Proposition 5.4]{Sa-Se} the constant coefficient Stokes system and $p=2$).
\end{proof}
\begin{remark}
\label{isotropic-case-1}
In the isotropic case \eqref{isotropic-1}, Definition $\ref{s-l-S-variational-variable-adj}$ reduces to Definition $\ref{s-l-S-variational-variable}$, and the single layer operator ${\boldsymbol{\mathcal V}}_{\partial\Omega}:B_{p,p}^{-\frac{1}{p}}(\partial \Omega )^n\to B_{p,p}^{1-\frac{1}{p}}(\partial \Omega )^n$ is self adjoint, i.e., formula \eqref{sa-sl} becomes
\begin{align}
\label{sa-sl-adj}
\left\langle \boldsymbol\psi ,{\mathcal V}_{\partial\Omega}\boldsymbol\psi^* \right\rangle _{\partial \Omega }
=\left\langle {\mathcal V}_{\partial\Omega}\boldsymbol\psi, \boldsymbol\psi^*\right\rangle _{\partial \Omega },\ \forall \, \boldsymbol\psi \in B_{p,p}^{-\frac{1}{p}}(\partial \Omega )^n,\, \boldsymbol\psi^* \in B_{p',p'}^{-\frac{1}{p'}}(\partial \Omega )^n.
\end{align}
\end{remark}
For a given operator $T:X\to Y$, we denote by ${\rm{Ker}}\left\{T:X\to Y\right\}:=\left\{x\in X:T(x)=0\right\}$ the null space of $T$.
Let $\boldsymbol \nu $ denote the outward unit normal to $\Omega $, which exists a.e. on $\partial \Omega $, and let
${\rm span}\{\boldsymbol \nu\} :=\{c \boldsymbol \nu :c\in {\mathbb R}\}$. 
For $p\in (1,\infty )$, consider the space
\begin{align}
\label{nu}
&B_{p,p;\boldsymbol \nu}^{\frac{1}{p'}}(\partial \Omega )^n:=\big\{\boldsymbol\Phi \in B_{p,p}^{\frac{1}{p'}}(\partial \Omega )^n:\langle \boldsymbol\Phi,\boldsymbol \nu \rangle _{\partial \Omega }=0\big\}.
\end{align}
Next we show main properties of the single layer operator (see also \cite[Lemma 4.9]{K-M-W-1} for $p=2$, \cite[Theorem 10.5.3]{M-W}, and \cite[Proposition 3.3(c)]{B-H}, \cite[Proposition 5.4]{Sa-Se} in the constant case).

Let us denote 
$\chi _{_{\Omega _{+}}}=
\left\{\begin{array}{ll}
1 & \mbox{ in }\ \Omega _+\\
0 & \mbox{ in }\ \Omega _-.
\end{array}
\right.$
\begin{lemma}
\label{slp-properties}
Let $\mathbb A$ satisfy conditions \eqref{Stokes-1} and \eqref{mu}, $p_*\in (2,\infty )$ be as in Lemma~\ref{lemma-a47-1-Stokes} and $p\in {\mathcal R}(p_*,n)$, cf. \eqref{p*}.
Then 
\begin{align}
\label{sl-n}
&{\bf V}_{\partial\Omega}{\boldsymbol\nu }={\bf 0} \mbox{ in } {\mathbb R}^n,\ {\mathcal Q}^s_{\partial\Omega}{\boldsymbol\nu }=-\chi _{_{\Omega _{+}}},\\
\label{sl-n-p}
&{{\mathcal V}_{\partial\Omega}\boldsymbol \nu ={\bf 0} \mbox{ a.e. on } \partial \Omega }\,,
\\
\label{range-sl-v}
&{{\mathcal V}_{\partial\Omega}\boldsymbol\psi \in B_{p,p;\boldsymbol \nu }^{1-\frac{1}{p}}(\partial \Omega )^n,\ \forall \, \boldsymbol\psi \in B_{p,p}^{-\frac{1}{p}}(\partial \Omega )^n}.
\end{align}

In addition, for any $p\in \left[2,p_*\right)\cap \left[2,n\right)$,
\begin{align}
\label{kernel-sl-v}
{\rm{Ker}}\big\{\boldsymbol{\mathcal V}_{\partial\Omega}:B_{p,p}^{-\frac{1}{p}}(\partial\Omega )^n\to B_{p,p}^{1-\frac{1}{p}}(\partial\Omega )^n\big\}={\rm span}\{\boldsymbol \nu\}.
\end{align}

\end{lemma}
\begin{proof}
First, note that 
Theorem $\ref{slp-var-apr-1-p}$ implies that the transmission problem \eqref{var-Brinkman-transmission-sl} with the datum $\boldsymbol\psi =\boldsymbol \nu \in B_{p,p}^{-\frac{1}{p}}(\partial \Omega )^n$ is well-posed. Moreover, the pair
$\left({\bf u}_{\boldsymbol \nu },\pi _{\boldsymbol \nu }\right)=\big({\bf 0},-\chi _{_{\Omega _{+}}}\big)\in {\mathcal H}^1_{p}({\mathbb R}^n)^n\times L_p({\mathbb R}^n)$
is the unique solution of this transmission problem. Then relations \eqref{sl-n} and \eqref{sl-n-p} follow from Definition \ref{s-l-S-variational-variable}. 
Thus, 
\begin{align}\label{sK}
{\rm span}\{\boldsymbol \nu\}
\subseteq {\rm{Ker}}\big\{{\boldsymbol{\mathcal V}}_{\partial\Omega}:B_{p,p}^{-\frac{1}{p}}(\partial\Omega )^n\to B_{p,p}^{1-\frac{1}{p}}(\partial\Omega )^n\big\}\quad
\forall\ p\in {\mathcal R}(p_*,n).
\end{align}
Similarly, 
\begin{align}
\label{sl-n-adj}
{\bf V}^*_{\partial\Omega}{\boldsymbol\nu }={\bf 0} \mbox{ in } {\mathbb R}^n,\ \
{{\mathcal V}^*_{\partial\Omega}\boldsymbol \nu ={\bf 0} \mbox{ a.e. on } \partial \Omega },
\end{align}
where ${\mathcal V}^*_{\partial\Omega}:B_{p',p'}^{-\frac{1}{p'}}(\partial\Omega )^n\to B_{p',p'}^{1-\frac{1}{p'}}(\partial\Omega )^n$ is the single layer operator for the adjoint Stokes system \eqref{var-Brinkman-transmission-sl-adj} (see Definition \ref{s-l-S-variational-variable-adj}).
By using formula \eqref{sa-sl} for the densities $\boldsymbol\psi \!\in \! B_{p,p}^{-\frac{1}{p}}(\partial\Omega )^n$ and $\boldsymbol\psi^*=\boldsymbol\nu \!\in \! B_{p',p'}^{-\frac{1}{p'}}(\partial\Omega )^n$, and the second relation in \eqref{sl-n-adj}, we obtain relation \eqref{range-sl-v}.

Next we determine the kernel of the single layer operator in case $p=2$. To do so, we assume that $\boldsymbol \psi _0\in {\rm{Ker}}\left\{{\boldsymbol{\mathcal V}}_{\partial\Omega}
:H^{-\frac{1}{2}}(\partial\Omega )^n\to H^{\frac{1}{2}}(\partial\Omega )^n\right\}$. 
Let
$({\bf u}_{\boldsymbol \psi _0},\pi_{\boldsymbol \psi _0})
=\big({\bf V}_{\partial \Omega}{\boldsymbol\psi _0},
{\mathcal Q}^s_{\partial\Omega}{\boldsymbol\psi _0}\big)$ 
be the unique solution in ${\mathcal H}^1({\mathbb R}^n)^n\times L_2({\mathbb R}^n)$ of the transmission problem \eqref{var-Brinkman-transmission-sl} with given datum $\boldsymbol \psi _0$. 
According to formula \eqref{jump-conormal-derivative-1f0} and the assumption that $\gamma {\bf u}_{\boldsymbol \psi _0}={\bf 0}$ a.e. on $\partial \Omega $, we obtain that
\begin{align}
a_{{\mathbb R}^n}\big({\bf u}_{\boldsymbol \psi _0},{\bf u}_{\boldsymbol \psi _0}\big)=\left\langle [{\bf T}({\bf u}_{\boldsymbol \psi _0},\pi_{\boldsymbol \psi _0})],\gamma {\bf u}_{\boldsymbol \psi _0}\right\rangle _{\partial \Omega }=0.
\end{align}
In addition, assumption \eqref{mu} yields that
$a_{{\mathbb R}^n}\big({\bf u}_{\boldsymbol \psi _0},{\bf u}_{\boldsymbol \psi _0}\big)\geq c_{\mathbb A}^{-1}\|\nabla ({\bf u}_{\boldsymbol \psi _0})\|_{L_2({\mathbb R}^n)^n}^2.$
Therefore, ${\bf u}_{\boldsymbol \psi _0}$ is a constant field, but the membership of ${\bf u}_{\boldsymbol \psi _0}$ in ${\mathcal H}^1({\mathbb R}^n)^n\hookrightarrow L_{\frac{2n}{n-2}}({\mathbb R}^n)^n$ shows that ${\bf u}_{\boldsymbol \psi _0}={\bf 0}$ in ${\mathbb R}^n$. 
Moreover, the Stokes equation satisfied by ${\bf u}_{\boldsymbol \psi _0}$ and $\pi _{\boldsymbol \psi _0}$ in ${\mathbb R}^n\setminus \partial \Omega $ and the membership of $\pi _{\boldsymbol \psi _0}$ in $L_2({\mathbb R}^n)$ show that $\pi _{\boldsymbol \psi _0}=c_0\chi _{_{\Omega _{+}}}$ in ${\mathbb R}^n$, where $c_0\in {\mathbb R}$.
Then formula \eqref{jump-conormal-derivative-1f0} and the divergence theorem yield that
$\langle [{\bf T}({\bf u}_{\boldsymbol \psi _0},\pi_{\boldsymbol \psi _0})],\gamma {\bf w}\rangle _{\partial \Omega }=-\langle \pi _{\boldsymbol \psi _0},{\rm{div}}\, {\bf w}\rangle _{{\mathbb R}^n}=-c_0\langle \boldsymbol \nu ,\gamma {\bf w}\rangle _{\partial \Omega },$ for any ${\bf w}\in {\mathcal D}({\mathbb R}^n)^n,$
and accordingly that $\boldsymbol \psi _0=[{\bf T}({\bf u}_{\boldsymbol \psi _0 },\pi_{\boldsymbol \psi _0})]=-c_0\boldsymbol \nu $. Hence, \eqref{kernel-sl-v} follows for $p=2$. 

Moreover, for any $p\in \left[2,p_*\right)\cap \left[2,n\right)$ by the inclusion 
$B_{p,p}^{-\frac{1}{p}}(\partial\Omega )^n\hookrightarrow H^{-\frac{1}{2}}(\partial\Omega )^n$
 we have 
$$
{\rm{Ker}}\big\{{\boldsymbol{\mathcal V}}_{\partial\Omega}:B_{p,p}^{-\frac{1}{p}}(\partial\Omega )^n\to B_{p,p}^{1-\frac{1}{p}}(\partial\Omega )^n\big\}\subseteq
{\rm{Ker}}\big\{{\boldsymbol{\mathcal V}}_{\partial\Omega}:H^{-\frac{1}{2}}(\partial\Omega )^n\to H^{\frac{1}{2}}(\partial\Omega )^n\big\}={\rm span}\{\boldsymbol \nu\}.
$$
Then by \eqref{sK} we conclude that \eqref{kernel-sl-v} holds also for any $p\in \left[2,p_*\right)\cap \left[2,n\right)$.
\end{proof}

Next we 
show the following property (see also \cite[Theorem 10.5.3]{M-W}, \cite[Proposition 3.3(d)]{B-H}, \cite[Proposition 5.5]{Sa-Se} in the constant case).
\begin{lemma}
\label{isom-sl-v}
Let $\mathbb A$ satisfy conditions \eqref{Stokes-1} and \eqref{mu}. Then the following operator is an isomorphism,
\begin{align}
\label{sl-v-isom}
\boldsymbol{\mathcal V}_{\partial\Omega}:H^{-\frac{1}{2}}(\partial\Omega )^n/{\rm span}\{\boldsymbol \nu\} \to H_{\boldsymbol \nu }^{\frac{1}{2}}(\partial\Omega )^n\,.
\end{align}
\end{lemma}
\begin{proof}
Let {$\left[\!\left[\cdot \right]\!\right]$} denote the classes in $H^{-\frac{1}{2}}(\partial\Omega )^n/{\rm span}\{\boldsymbol \nu\}$, $\left[\!\left[\boldsymbol \psi \right]\!\right]\!=\!\boldsymbol \psi \!+\!{\rm span}\{\boldsymbol \nu\} $, with $\boldsymbol \psi \in H^{-\frac{1}{2}}(\partial\Omega )^n$.
The invertibility is based on the coercivity inequality
\begin{align}
\label{H-elliptic-sl-v}
\left\langle \left[\!\left[\boldsymbol \psi \right]\!\right],\boldsymbol{\mathcal V}_{\partial\Omega}\left[\!\left[\boldsymbol \psi \right]\!\right]\right\rangle _{\partial \Omega }\geq c\left\|\left[\!\left[\boldsymbol \psi \right]\!\right]\right\|_{H^{-\frac{1}{2}}(\partial\Omega )^n/{{\rm span}\{\boldsymbol \nu\}}}^2,\ \forall \, \left[\!\left[\boldsymbol \psi \right]\!\right]\in {H^{-\frac{1}{2}}(\partial\Omega )^n/{{\rm span}\{\boldsymbol \nu\}}},
\end{align}
which follows
by the arguments similar to those in \cite[Lemma 4.10]{K-M-W-1} and \cite[Proposition 5.5]{Sa-Se}. 
Indeed, 
according to formula \eqref{jump-conormal-derivative-1f0}, Definition \ref{s-l-S-variational-variable}, relations \eqref{range-sl-v}, \eqref{kernel-sl-v}, and inequality \eqref{a-1-v2-S}, we obtain that
\begin{align}
\label{onto-div-2}
\left\langle \left[\!\left[\boldsymbol \psi \right]\!\right],\boldsymbol{\mathcal V}_{\partial\Omega}\left[\!\left[\boldsymbol \psi \right]\!\right]\right\rangle _{\partial \Omega }&=\left\langle \boldsymbol \psi ,\boldsymbol{\mathcal V}_{\partial\Omega}\boldsymbol \psi \right\rangle _{\partial \Omega }=\langle [{\bf T}({\bf u}_{\boldsymbol \psi},\pi_{\boldsymbol \psi})],\gamma {\bf u}_{\boldsymbol \psi}\rangle _{\partial \Omega }\nonumber\\
&=a_{\mathbb R^n}({\bf u}_{\boldsymbol \psi },{\bf u}_{\boldsymbol \psi })
\geq c_{\mathbb A}^{-1}c_1\|{\bf u}_{\boldsymbol \psi }\|_{\mathcal H^1({\mathbb R}^n)^n}^2,
\end{align}
where ${\bf u}_{\boldsymbol \psi }\!=\!{\bf V}_{\partial\Omega}\boldsymbol \psi $, $\pi_{\boldsymbol\psi}\!=\!{\mathcal Q}^s_{\partial\Omega}{\boldsymbol\psi}$.
Since the trace operator
$\gamma :{\mathcal H}_{{\rm div}}^1({\mathbb R}^n)^n\!\to \!H_{\boldsymbol \nu }^{\frac{1}{2}}(\partial\Omega )^n$
is surjective with a bounded right inverse $\gamma ^{-1}:H_{\boldsymbol \nu }^{\frac{1}{2}}(\partial\Omega )^n\!\to \!{\mathcal H}_{{\rm div}}^1({\mathbb R}^n)^n$ (cf., e.g., \cite[Proposition 4.4]{Sa-Se}), for any $\boldsymbol \Phi \!\in \!H_{\boldsymbol \nu }^{\frac{1}{2}}(\partial\Omega )^n$ we have the inclusion 
${\bf w}:=\gamma ^{-1}\boldsymbol \Phi \in {\mathcal H}_{{\rm div}}^1({\mathbb R}^n)^n$. 
Then there exists a constant $c'\!=\!c'(\partial \Omega ,n)\!>\!0$ such that
\begin{align}
\label{estimate-norm}
|\langle \left[\!\left[\boldsymbol \psi \right]\!\right],\boldsymbol \Phi \rangle _{\partial \Omega }|&=|\langle \boldsymbol \psi ,\boldsymbol \Phi \rangle _{\partial \Omega }|=|\langle [{\bf T}({\bf u}_{\boldsymbol \psi },\pi_{\boldsymbol \psi })],\gamma {\bf w}\rangle _{\partial \Omega }|=|a_{\mathbb R^n}({\bf u}_{\boldsymbol \psi },{\bf w})|\nonumber\\
&\!\leq \!c_{\mathbb A}\|{\bf u}_{\boldsymbol \psi }\|_{{\mathcal H}^1({\mathbb R}^n)^n}\|\gamma ^{-1}\boldsymbol \Phi \|_{{\mathcal H}^1({\mathbb R}^n)^n}\leq c_{\mathbb A}c'\|{\bf u}_{\boldsymbol \psi }\|_{{\mathcal H}^1({\mathbb R}^n)^n}\|\boldsymbol \Phi \|_{H^{\frac{1}{2}}(\partial\Omega )^n}\,.
\end{align}
Then formula \eqref{estimate-norm} and the duality of the spaces 
$H_{\boldsymbol \nu }^{\frac{1}{2}}(\partial\Omega )^n$ and 
$H^{-\frac{1}{2}}(\partial\Omega )^n/{{\rm span}\{\boldsymbol \nu\}}$ imply that
\begin{align}
\label{onto-div-1}
\|\left[\!\left[\boldsymbol \psi\right]\!\right]\|_{H^{-\frac{1}{2}}(\partial\Omega )^n/{\rm span}\{\boldsymbol \nu\}}
\leq c_{\mathbb A}c'\|{\bf u}_{\boldsymbol \psi }\|_{{\mathcal H}^{1}({\mathbb R}^n)^n}.
\end{align}
Then inequality \eqref{H-elliptic-sl-v} follows from inequalities \eqref{onto-div-2} and \eqref{onto-div-1}. Finally, the Lax-Milgram lemma implies that the single layer potential operator \eqref{sl-v-isom} is an isomorphism, as asserted.
\end{proof}
\subsection{The double layer potential operator for the Stokes system with $L_{\infty }$ coefficients}
Next we present the well-posedness results for a transmission problem used for the definition of 
$L_{\infty }$-coefficient Stokes double layer potentials in the space $B_{p,p}^{1-\frac{1}{p}}(\partial \Omega )^n$ with
$p$ in some open set containing $2$ and $n\geq 3$
(cf. \cite[Propositions 6.1, 7.1]{Sa-Se} in the case $n=2,3$, $p=2$ and ${\mathbb A}=\mathbb I$). 
Recall that if  
$\mathbf u\in L_{p,\rm loc}(\mathbb R^n)^n$ is such that $\mathbf u|_{\Omega_+}\in {H}^1_{p}(\Omega _+)^n$, 
$\mathbf u|_{\Omega_-}\in  {\mathcal H}^1_{p}(\Omega _{-})^n$, we will denote this as 
$\mathbf u\in\mathcal H^1_p(\mathbb R^n\setminus\partial\Omega)$ and employ the norm
$\|u\|^p_{\mathcal H^1_p(\mathbb R^n\setminus\partial\Omega)}:=\|u\|^p_{H^1_p(\Omega_+)}+\|u\|^p_{\mathcal H^1_p(\Omega_-)}$.
\begin{theorem}
\label{dlp-var-apr-var-p}
Let $\mathbb A$ satisfy conditions \eqref{Stokes-1} and \eqref{mu}, $p_*\in (2,\infty )$ be as in Lemma~\ref{lemma-a47-1-Stokes} and $p\in {\mathcal R}(p_*,n)$, cf. \eqref{p*}.
Then for any $\boldsymbol\varphi\in B_{p,p }^{1-\frac{1}{p}}(\partial \Omega )^n$, the transmission problem
\begin{equation}
\label{transmission-B-dl-var}
\left\{
\begin{array}{lll}
\partial _\alpha\left(A^{\alpha\beta}\partial_\beta({\bf u})\right)-\nabla \pi={\bf 0} & \mbox{ in } \mathbb R^n\setminus \partial\Omega \,,\\
{\rm{div}}\, {\bf u}=0 & \mbox{ in } \mathbb R^n\setminus \partial\Omega \,,
\\
\left[\gamma ({\bf u})\right]= -\boldsymbol\varphi,\quad 
\left[{\bf T}({\bf u},\pi)\right]
={{\bf 0}} & \mbox{ on } \partial \Omega \,,
\end{array}\right.
\end{equation}
has a unique solution 
$({\bf u}_{\boldsymbol\varphi},\pi_{\boldsymbol\varphi})\in {\mathcal H}^1_{p}(\mathbb R^n\setminus\partial\Omega)^n
\times L_p(\mathbb R^n)$, 
and there exists a constant $C=C(\partial \Omega ,c_{\mathbb A},p,n)>0$ such that
\begin{align*}
\|{\bf u}_{{\boldsymbol\varphi}}\|_{{\mathcal H}^1_{p}(\mathbb R^n\setminus\partial\Omega)^n}
+\|\pi _{\boldsymbol\varphi}\|_{L_p(\mathbb R^n)}
\leq C\|{\boldsymbol\varphi}\|_{B_{p,p}^{1-\frac{1}{p}}(\partial \Omega )^n}.
\end{align*}
\end{theorem}
\begin{proof}
Let 
$p\!\in \!{\mathcal R}(p_*,n)$ and
$\boldsymbol\varphi\in B_{p,p }^{1-\frac{1}{p}}(\partial \Omega )^n$.
First we show uniqueness.
Let 
$({\bf u}_{0},\pi_{0})\in{\mathcal H}^1_{p}(\mathbb R^n\setminus\partial\Omega)^n\times L_p(\mathbb R^n)$
be a solution of the homogeneous version of problem \eqref{transmission-B-dl-var}. 
Then the first transmission condition implies that ${\bf u}_0\in {\mathcal H}^1_{p}({\mathbb R}^n)^n$.  
Hence, $({\bf u}_0,\pi _0)\in {\mathcal H}^1_{p}({\mathbb R}^n)^n\times L_p({\mathbb R}^n)$ is a solution of the homogeneous version of the transmission problem \eqref{var-Brinkman-transmission-sl}, which, in view of Theorem \ref{slp-var-apr-1-p}, has only the trivial solution. 

The arguments similar to the ones for Theorem \ref{slp-var-apr-1-p} imply that problem \eqref{transmission-B-dl-var} has the following equivalent variational formulation:

{\it Find $({\bf u}_{\boldsymbol\varphi},\pi_{\boldsymbol\varphi})\in {\mathcal H}^1_{p}(\mathbb R^n\setminus\partial\Omega)^n\times L_p(\mathbb R^n)$
such that}
\begin{equation}
\label{transmission-B-dl-var-equiv}
\left\{\begin{array}{lll}
\left \langle A^{\alpha \beta }\partial_\beta {\bf u}_{\boldsymbol\varphi},\partial_\alpha {\bf v}\right \rangle_{\Omega_+\cup\,\Omega_-}
-\langle \pi_{\boldsymbol\varphi},{\rm{div}}\, {\bf v}\rangle _{{\mathbb R}^n} 
=0,\ \forall \, {\bf v}\in {\mathcal H}^1_{p'}({\mathbb R}^n)^n,\\
\langle {\rm{div}}\, {\bf u}_{{\boldsymbol\varphi}},q\rangle_{\Omega_+\cup\,\Omega_-}=0,\ \forall \, q\in L_{p'}({\mathbb R}^n),\\
\left[\gamma ({\bf u}_{{\boldsymbol\varphi}})\right]= -\boldsymbol\varphi \mbox{ on } \partial \Omega.
\end{array}
\right.
\end{equation}
The existence of the bounded right inverses 
$\gamma^{-1}_\pm: B_{p,p}^{1-\frac{1}{p}}({\partial\Omega })\to\mathcal H^1_{p}(\Omega_\pm)$ to the trace operators 
$\gamma\pm \!:\!{\mathcal H}^1_{p}(\Omega_\pm)\!\to \! B_{p,p}^{1-\frac{1}{p}}(\partial\Omega)$ implies that for $\boldsymbol\varphi\in B_{p,p}^{1-\frac{1}{p}}(\partial \Omega )^n$ given,
there is ${\bf w}_{\boldsymbol\varphi}\in {\mathcal H}^1_{p}(\Omega_\pm)^n$, such that
$\left[\gamma {\bf w}_{\boldsymbol\varphi}\right]=-{\boldsymbol\varphi} \mbox{ on } \partial \Omega .$
Thus, ${\bf v}_{{\boldsymbol\varphi}}:={\bf u}_{{\boldsymbol\varphi}}-{\bf w}_{{\boldsymbol\varphi}}$ has no jump across $\partial \Omega $,
and hence ${\bf v}_{{\boldsymbol\varphi}}\in {\mathcal H}^1_{p}({\mathbb R}^n)^n$ (see also \cite[Theorem 5.13]{B-M-M-M}). Moreover, problem \eqref{transmission-B-dl-var-equiv} reduces to the variational problem
\begin{align}
\label{transmission-B-variational-dl-3-equiv-var}
\left\{\begin{array}{ll}
a_{{\mathbb R}^n}({\bf v}_{\boldsymbol\varphi},{\bf v})+b_{\mathbb R^n}({\bf v},\pi_{\boldsymbol\varphi})
=\boldsymbol\xi_{\boldsymbol\varphi}({\bf v}),\ \forall \, {\bf v}\in {\mathcal H}^1_{p'}({\mathbb R}^n)^n\,,\\
b_{\mathbb R^n}({\bf v}_{\boldsymbol\varphi},q)
=\zeta_{\boldsymbol\varphi}(q),\ \forall \, q\in L_{p'}(\mathbb R^n)\,,
\end{array}
\right.
\end{align}
with the unknown $({\bf v}_{{\boldsymbol\varphi}},\pi _{{\boldsymbol\varphi}})\!\in \!{\mathcal H}^1_{p}({\mathbb R}^n)^n\!\times \!L_p({\mathbb R}^n)$, where $a_{{\mathbb R}^n}:{\mathcal H}^1_{p}({\mathbb R}^n)^n\times {\mathcal H}^1_{p'}({\mathbb R}^n)^n\to {\mathbb R}$ and $b_{\mathbb R^n}: {\mathcal H}^1_{p}({\mathbb R}^n)^n\times L_{p'}({\mathbb R}^n)\to {\mathbb R}$ are the bounded bilinear forms given by \eqref{a-v} and \eqref{b-v}, respectively. 
In addition, conditions \eqref{Stokes-1} show the boundedness of the linear forms
\begin{align}
\boldsymbol\xi _{{\boldsymbol\varphi}}:{\mathcal H}^1_{p'}({\mathbb R}^n)^n\to {\mathbb R},\ \boldsymbol\xi _{{\boldsymbol\varphi}}({\bf v})
:=-\big\langle A^{\alpha \beta }\partial _\beta {\bf w}_{\boldsymbol\varphi },\partial _\alpha {\bf v}\big\rangle _{\Omega _+}-\big\langle A^{\alpha \beta }\partial _\beta {\bf w}_{\boldsymbol\varphi },\partial _\alpha {\bf v}\big \rangle _{\Omega _-},
\\
\zeta _{\boldsymbol\varphi}:L_{p'}({\mathbb R}^n)^n\to {\mathbb R},\ \zeta _{\boldsymbol\varphi}(q)
:=-\langle {\rm{div}}\, {\bf w}_{{{\boldsymbol\varphi}}},q\rangle _{\Omega _+}
-\langle {\rm{div}}\, {\bf w}_{{{\boldsymbol\varphi}}},q\rangle _{\Omega _{-}},\ 
\forall \, q\in L_{p'}({\mathbb R}^n)\,.
\end{align}
Therefore, Lemma \ref{lemma-a47-1-Stokes} shows that the variational problem \eqref{transmission-B-variational-dl-3-equiv-var} 
has a unique solution $({\bf v}_{{\boldsymbol\varphi}},\pi _{{\boldsymbol\varphi}})\in {\mathcal H}^1_{p}({\mathbb R}^n)^n\times L_{p}({\mathbb R}^n)$. 
Then the pair
$({\bf u}_{{\boldsymbol\varphi}},\pi _{{\boldsymbol\varphi}})=({\bf w}_{{\boldsymbol\varphi}}+{\bf v}_{{\boldsymbol\varphi}},\pi _{{\boldsymbol\varphi}})$
is a solution of the variational problem \eqref{transmission-B-dl-var-equiv} in 
${\mathcal H}^1_{p}(\mathbb R^n\setminus\partial\Omega)^n\times L_p({\mathbb R}^n)$, 
and due to the equivalence between problems \eqref{transmission-B-dl-var} and  \eqref{transmission-B-dl-var-equiv}, it is also the unique solution of the problem \eqref{transmission-B-dl-var} in 
${\mathcal H}^1_{p}(\mathbb R^n\setminus\partial\Omega)^n\times L_p({\mathbb R}^n)$. 
\end{proof}

Theorem \ref{dlp-var-apr-var-p} leads to the following definition of the double layer operator for the nonsmooth-coefficient Brinkman system \eqref{Stokes} (cf. \cite[p. 77]{Sa-Se} for the constant-coefficient Stokes system in ${\mathbb R}^3$, and \cite[formula (4.5) and Lemma 4.6]{Barton} for general strongly elliptic differential operators).
\begin{definition}
\label{d-l-variational-var}
Let $\mathbb A$ satisfy conditions \eqref{Stokes-1} and \eqref{mu}, $p_*\in (2,\infty )$ be as in Lemma~\ref{lemma-a47-1-Stokes} and  $p\in {\mathcal R}(p_*,n)$, cf. \eqref{p*}.
For any $\boldsymbol\varphi \in B_{p,p}^{1-\frac{1}{p}}(\partial \Omega )^n$,
we define the {\it double layer potentials} with the density $\boldsymbol\varphi $ for the Stokes operator $\boldsymbol{\mathcal L}$ with coefficients 
${\mathbb A}$ as
\begin{align*}
{\bf W}_{\partial\Omega}{\boldsymbol\varphi }:={\bf u}_{\boldsymbol\varphi },\, 
{\mathcal Q}^d_{\partial\Omega}{\boldsymbol\varphi}:=\pi _{\boldsymbol\varphi },
\end{align*}
and the boundary operators 
${\bf K}_{\partial\Omega}:B_{p,p}^{1-\frac{1}{p}}(\partial\Omega )^n\to B_{p,p}^{1-\frac{1}{p}}(\partial\Omega )^n$
and
${\bf D}_{\partial\Omega}:B_{p,p}^{1-\frac{1}{p}}(\partial\Omega )^n\to B_{p,p}^{-\frac{1}{p}}(\partial\Omega )^n$ as
\begin{align}
\label{dlp-oper-var}
{{\bf K}_{\partial\Omega}{\boldsymbol\varphi }:=\frac{1}{2}\left(\gamma _{+}{\bf u}_{\boldsymbol\varphi }+\gamma _{-}{\bf u}_{\boldsymbol\varphi }\right)},\quad
{\bf D}_{\partial\Omega}\boldsymbol\varphi:={\bf T}^{+}\left({\bf W}_{\partial\Omega}\boldsymbol\varphi ,{\mathcal Q}^d_{\partial\Omega}\boldsymbol\varphi \right)={\bf T}^{-}\left({\bf W}_{\partial\Omega}\boldsymbol\varphi ,{\mathcal Q}^d_{\partial\Omega}\boldsymbol\varphi \right), 
\end{align}
where $({\bf u}_{\boldsymbol\varphi },\pi _{\boldsymbol\varphi })$ is the unique solution of the transmission problem \eqref{transmission-B-dl-var} in ${\mathcal H}^1_{p}(\mathbb R^n\setminus\partial\Omega)^n\times L_p({\mathbb R}^n)$.
\end{definition}
Theorem \ref{dlp-var-apr-var-p} and Definition \ref{d-l-variational-var} lead to the next result (see also \cite[(10.81), (10.82)]{M-W} and \cite[Propositions 6.2, 6.3]{Sa-Se} for the constant coefficient Stokes system in ${\mathbb R}^3$, and {\cite[Lemma 5.8]{Barton}}).
\begin{lemma}
\label{continuity-dl-h-var}
Let $\mathbb A$ satisfy conditions \eqref{Stokes-1} and \eqref{mu}, $p_*\in (2,\infty )$ be as in Lemma~\ref{lemma-a47-1-Stokes} and  $p\in {\mathcal R}(p_*,n)$, cf. \eqref{p*}.
\begin{itemize}
\item[$(i)$]
The following  operators are linear and continuous,
\begin{align}
\label{dl-var-mu-alpha}
&{\bf W}_{\partial\Omega}:B_{p,p}^{1-\frac{1}{p}}(\partial \Omega )^n\to {\mathcal H}^1_{p}(\mathbb R^n\setminus\partial\Omega)^n,\quad
{\mathcal Q}_{\partial\Omega}^d:B_{p,p}^{1-\frac{1}{p}}(\partial \Omega )^n\to {L_p({\mathbb R}^n)},\\
\label{dl-var-mu-alpha-1}
&{\bf K}_{\partial\Omega}:B_{p,p}^{1-\frac{1}{p}}(\partial \Omega )^n\to B_{p,p}^{1-\frac{1}{p}}(\partial \Omega )^n,
\quad
{\bf D}_{\partial\Omega}:B_{p,p}^{1-\frac{1}{p}}(\partial \Omega )^n\to B_{p,p}^{-\frac{1}{p}}(\partial \Omega )^n\,.
\end{align}
\item[$(ii)$]
For any $\boldsymbol\varphi \in B_{p,p}^{1-\frac{1}{p}}(\partial \Omega )^n$, the following jump relations hold a.e. on $\partial \Omega $
\begin{align}
\label{jump-dl-v-var-alpha}
& 
\gamma _{\pm}{\bf W}_{\partial\Omega}\boldsymbol\varphi 
=\mp \frac{1}{2}\boldsymbol\varphi +{\bf K}_{\partial\Omega}\boldsymbol\varphi ,\quad
{\bf T}^\pm\left({\bf W}_{\partial\Omega}\boldsymbol\varphi ,{\mathcal Q}^d_{\partial\Omega}\boldsymbol\varphi \right)
={\bf D}_{\partial\Omega}\boldsymbol\varphi \,.
\end{align}
\item[$(iii)$]
The operator 
${\mathcal K}^*_{\partial\Omega}:B_{p',p'}^{-\frac{1}{p'}}(\partial \Omega )^n\to B_{p',p'}^{-\frac{1}{p'}}(\partial \Omega )^n$ 
defined in
\eqref{slp-S-oper-var-adjoint} is transpose to the operator 
${\bf K}_{\partial\Omega}:B_{p,p}^{1-\frac{1}{p}}(\partial\Omega )^n\to B_{p,p}^{1-\frac{1}{p}}(\partial\Omega )^n$
defined in \eqref{dlp-oper-var}, i.e.,
\begin{align}
\label{transpose-dl-var-sm}
\left\langle \boldsymbol\psi^* ,{\bf K}_{\partial\Omega}\boldsymbol\varphi \right\rangle _{\partial \Omega}=\left\langle {\mathcal K}^*_{\partial\Omega}\boldsymbol\psi^* ,\boldsymbol\varphi \right\rangle _{\partial \Omega},\, \forall \, \boldsymbol\varphi\in B_{p,p}^{1-\frac{1}{p}}(\partial \Omega )^n,\, \boldsymbol\psi^*\in B_{p',p'}^{-\frac{1}{p'}}(\partial \Omega )^n\,.
\end{align}
\end{itemize}
\end{lemma}
\begin{proof}
The continuity of operators \eqref{dl-var-mu-alpha} and \eqref{dl-var-mu-alpha-1} follows from well-posedness of the transmission problem \eqref{transmission-B-dl-var} and Definition \ref{d-l-variational-var}. 
Moreover, the transmission conditions in \eqref{transmission-B-dl-var} and again Definition \ref{d-l-variational-var} lead to the jump formulas \eqref{jump-dl-v-var-alpha}. 

Next we show equality \eqref{transpose-dl-var-sm} using an argument similar to that in the proof of \cite[Proposition 6.7]{Sa-Se} for the constant-coefficient Stokes system and $p=2$.
Let $\boldsymbol\varphi\in B_{p,p}^{1-\frac{1}{p}}(\partial \Omega )^n$ be given, and let $({\bf u}_{\boldsymbol\varphi },\pi _{\boldsymbol\varphi })=\big({\bf W}_{\partial\Omega}\boldsymbol\varphi ,{\mathcal Q}_{\partial\Omega}^d\boldsymbol\varphi \big)\in B_{p,p}^{1-\frac{1}{p}}(\partial \Omega )^n\times L_p({\mathbb R}^n)$ be the unique solution of problem \eqref{transmission-B-dl-var}.
Let also $\boldsymbol\psi^* \in B_{p',p'}^{-\frac{1}{p'}}(\partial \Omega )^n$. 
According to formulas \eqref{jump-conormal-derivative-1-adj} and \eqref{jump-dl-v-var-alpha}, 
\begin{align}
\label{dl-identity-var}
0=\langle [{\bf T}({\bf W}_{\partial\Omega}\boldsymbol\varphi ,{\mathcal Q}_{\partial\Omega}^d\boldsymbol\varphi )],\gamma {\bf V}^*_{\partial\Omega}\boldsymbol\psi^* \rangle _{\partial \Omega }
=\big \langle A^{\alpha \beta }\partial _\beta \big({\bf W}_{\partial\Omega}\boldsymbol\varphi \big),\partial _\alpha \big({\bf V}^*_{\partial\Omega}\boldsymbol\psi^* \big)\big \rangle _{\Omega _+\cup \Omega _{-}}\,.
\end{align}
Then the Green identities \eqref{Green-particular-p-adj} and equality \eqref{dl-identity-var} yield that
\begin{align}
\label{sl-dl-var}
\!\!\!\!\!\!\!\!\left\langle {\bf T}^{*+}\left({\bf V}^*_{\partial\Omega}\boldsymbol\psi^* ,{\mathcal Q}^{s*}_{\partial\Omega}\boldsymbol\psi^* \right),\gamma _{+}({\bf W}_{\partial\Omega}\boldsymbol\varphi )\right\rangle _{\partial \Omega }\!\!=\!\!\left\langle {\bf T}^{*-}\left({\bf V}^*_{\partial\Omega}\boldsymbol\psi^* ,{\mathcal Q}^*_{\partial\Omega}\boldsymbol\psi^* \right),\gamma _{-}\left({\bf W}_{\partial\Omega}\boldsymbol\varphi \right)\right\rangle _{\partial \Omega }.
\end{align}
The second formula in  \eqref{sl-identities-var-1-adjoint},
the first formula in \eqref{jump-dl-v-var-alpha}, and relation \eqref{sl-dl-var} lead to equality \eqref{transpose-dl-var-sm}.
\end{proof}
We now show the following invertibility property of the operator ${\bf D}_{\partial\Omega}$ defined in \eqref{dl-var-mu-alpha-1} (see \cite[Propositions 6.4 and 6.5]{Sa-Se} in the constant-coefficient case).

For $s\in[-1,1]$, let us define the subspaces
$
H_{**}^{s}(\partial\Omega )^n:=\big\{\boldsymbol\Psi \in H^{s}(\partial\Omega )^n: \langle \boldsymbol\Psi ,{\bf 1}\rangle _{\partial \Omega }=0
$

\begin{lemma}
\label{isom-dlc}
Let $\mathbb A$ satisfy conditions \eqref{Stokes-1} and \eqref{mu}.
Then 
\begin{align}
\label{kernel-range-dl}
&{\rm Ker}\, \{{\bf D}_{\partial\Omega}:H^{\frac{1}{2}}(\partial\Omega )^n\to H^{-\frac{1}{2}}(\partial\Omega )^n\}={\mathbb R}^n, \\ 
&{\bf D}_{\partial\Omega}\,\boldsymbol\varphi \in H_{**}^{-\frac{1}{2}}(\partial\Omega )^n\quad \forall\, \boldsymbol\varphi \in H^{\frac{1}{2}}(\partial\Omega )^n \,,
\end{align}
and the following operator 
is an isomorphism,
\begin{align}
\label{hyp-v-isom}
{\bf D}_{\partial\Omega}:H_{**}^{\frac{1}{2}}(\partial\Omega )^n\to H_{**}^{-\frac{1}{2}}(\partial\Omega )^n .
\end{align}
\end{lemma}
\begin{proof}
$(i)$ First, we determine the kernel of the operator ${\bf D}_{\partial\Omega}:H^{-\frac{1}{2}}(\partial\Omega )^n\to H^{\frac{1}{2}}(\partial\Omega )^n$. Thus, assume that $\boldsymbol\varphi \in H^{\frac{1}{2}}(\partial\Omega )^n$ satisfies the equation ${\bf D}_{\partial\Omega}\boldsymbol\varphi ={\bf 0}$ on $\partial \Omega $, and use the notation ${\bf u}_{\boldsymbol\varphi} :={\bf W}_{\partial\Omega}\boldsymbol\varphi $ and $\pi _{\boldsymbol\varphi} :={\mathcal Q}_{\partial\Omega}^d\boldsymbol\varphi $. 
Then  jump relations \eqref{jump-dl-v-var-alpha}, the first Green identity \eqref{Green-particular-p} in Lemma \ref{lem-add1}, and assumption \eqref{mu} imply that
$\nabla({\bf u}_{\boldsymbol\varphi} )=0$ in $\Omega _\pm$. Then there exists a constant ${\bf b}\in {\mathbb R}^n$ such that ${\bf u}_{\boldsymbol\varphi} ={\bf b}$ in $\Omega _+$
and the inclusion
${\bf u}_{\boldsymbol\varphi} \in {\mathcal H}^1(\Omega _{-})^n\hookrightarrow L_{\frac{2n}{n-2}}({\Omega _{-}})^n$ implies that ${\bf u}_{\boldsymbol\varphi} ={\bf 0}$ in $\Omega _{-}$. 
Then by using again the jump relations \eqref{jump-dl-v-var-alpha} we deduce that $\boldsymbol\varphi =-{\bf b}$. 

Let ${\bf c}\in {\mathbb R}^n$ and let 
${\bf u}_{\bf c}:=-{\bf c}\chi _{\Omega _+}$,
$\pi _{\bf c}:=0$ in ${\mathbb R}^n.$
Then the pair $({\bf u}_{\bf c},\pi _{\bf c})$ belongs to ${\mathcal H}^1(\Omega_\pm)^n\times L_2({\mathbb R}^n)$ and satisfies transmission problem \eqref{transmission-B-dl-var} with  $\boldsymbol\varphi={\bf c}$. 
Then Definition \ref{d-l-variational-var} yields that ${\bf W}_{\partial\Omega}({\bf c})={\bf u}_{\bf c}$ and ${\mathcal Q}^d_{\partial\Omega}({\bf c})=0$ in ${\mathbb R}^n$, and by the second formula in \eqref{dlp-oper-var} we obtain ${\bf D}_{\partial\Omega}({\bf c})={\bf 0}$ on $\partial \Omega $.
Therefore, ${\rm{Ker }}\, {\bf D}_{\partial\Omega}={\mathbb R}^n$.

Now let $\boldsymbol\varphi \in H^{\frac{1}{2}}(\partial\Omega )^n$. 
By applying the first Green identity \eqref{Green-particular-p} to the pair 
$(\mathbf u,\pi)=({\bf W}_{\partial\Omega}\boldsymbol\varphi ,{\mathcal Q}_{\partial\Omega}^d\boldsymbol\varphi )$ and $\mathbf w=-\chi _{\Omega _+}$ and by using the second jump relation in \eqref{jump-dl-v-var-alpha}, we obtain that
${\langle {\bf D}_{\partial\Omega}\boldsymbol\varphi ,{\bf 1}\rangle _{\partial \Omega }=0}$, 
and hence the membership of ${\bf D}_{\partial\Omega}\boldsymbol\varphi $ in $H_{**}^{-\frac{1}{2}}(\partial\Omega )^n$.

$(ii)$
Next, we show the invertibility of operator \eqref{hyp-v-isom}. First, we note that relations \eqref{kernel-range-dl} imply that this operator is injective on the closed subspace $H_{**}^{\frac{1}{2}}(\partial\Omega )^n$ of $H^{\frac{1}{2}}(\partial\Omega )^n$, and that its range is a subset of $H_{**}^{-\frac{1}{2}}(\partial\Omega )^n$. Moreover, we assert that there is ${\mathcal C}\!=\!{\mathcal C}(\partial \Omega ,c_{\mathbb A},n)\!>\!0$ such that
\begin{align}
\label{coercive-dlc}
\langle -{\bf D}_{\partial\Omega}\boldsymbol\varphi ,\boldsymbol\varphi \rangle _{\partial \Omega}\geq {\mathcal C}\|\boldsymbol\varphi \|_{H^{\frac{1}{2}}(\partial \Omega )^n}^2,\ \forall \, \boldsymbol\varphi \in H_{**}^{\frac{1}{2}}(\partial\Omega )^n
\end{align}
(see also \cite[Proposition 6.5]{Sa-Se} in the constant coefficient case).
To this end, $\boldsymbol\varphi \in H_{0}^{\frac{1}{2}}(\partial\Omega )^n$ and we apply the first Green identity \eqref{Green-particular-p} to the pair 
$({\bf u}_{\boldsymbol\varphi} ,\pi_{\boldsymbol\varphi}):=
({\bf W}_{\partial\Omega}\boldsymbol\varphi,{\mathcal Q}_{\partial\Omega}^d\boldsymbol\varphi )$ 
and $\mathbf w=\mathbf u_{\boldsymbol\varphi}={\bf W}_{\partial\Omega}\boldsymbol\varphi$, and use the jump relations \eqref{jump-dl-v-var-alpha} and conditions \eqref{mu} 
to obtain the inequality
\begin{align}
\label{coercive-dlc-5}
\langle -{\bf D}_{\partial\Omega}\boldsymbol\varphi ,\boldsymbol\varphi \rangle _{\partial \Omega }\geq c_{\mathbb A} ^{-1}\|\nabla ({\bf u}_{\boldsymbol\varphi})\|_{L_2(\Omega_+\cup\,\Omega_-)^{n\times n}}^2\,.
\end{align}

On the other hand, the continuity of the trace operators $\gamma _\pm :{\mathcal H}^1(\Omega_\pm)^n\to H^{\frac{1}{2}}(\partial \Omega )^n$ and the first in jump relations \eqref{jump-dl-v-var-alpha} imply that there exists a constant ${\mathcal C}_1={\mathcal C}_1(\partial \Omega ,c_{\mathbb A},n)>0$ such that
\begin{align}
\label{coercive-dlc-1}
\!\!\!\!\!\|\boldsymbol\varphi \|_{H^{\frac{1}{2}}(\partial \Omega )^n}^2=\left\|\, [\gamma {\bf u}_{\boldsymbol\varphi} ]\, \right\|_{H^{\frac{1}{2}}(\partial \Omega )^n}^2
\leq {\mathcal C}_1\|{\bf u}_{\boldsymbol\varphi}\|_{{\mathcal H}^1(\mathbb R^n\setminus\partial\Omega)^n}^2\,.
\end{align}
Note that 
the formula
\begin{align}
\label{coercive-dlc-2}
{|\!|\!|{\bf v}|\!|\!|^2:=\|\nabla {\bf v}\|_{L_2(\Omega_+\cup\,\Omega_-)^{n\times n}}^2
+\left|\int_{\partial \Omega }[\gamma {\bf v}]d\sigma \right|^2, \ \forall \, 
{\bf v}\in {\mathcal H}^1(\mathbb R^n\setminus\partial\Omega)^n}
\end{align}
defines a norm on ${\mathcal H}^1(\mathbb R^n\setminus\partial\Omega)^n$ 
equivalent to the norm $\|\cdot \|_{{\mathcal H}^1({\mathbb R}^n\setminus \partial \Omega )^n}$ 
(see Lemma \ref{equiv-norm-Sobolev}).
Thus,
\begin{align}
\label{coercive-dlc-3}
\|{\bf v}\|_{{\mathcal H}^1(\mathbb R^n\setminus\partial\Omega)^n}
\leq {\mathcal C}_2|\!|\!|{\bf v}|\!|\!|,\ \forall \, {\bf v}\in {\mathcal H}^1(\mathbb R^n\setminus\partial\Omega)^n\,,
\end{align}
with some constant ${\mathcal C}_2>0$.
On the other hand, by choosing ${\bf v}={\bf u}_{\boldsymbol\varphi} $ in \eqref{coercive-dlc-2} and using again the jump formulas \eqref{jump-dl-v-var-alpha} and the assumption that $\boldsymbol\varphi \in H_{**}^{\frac{1}{2}}(\partial\Omega )^n$ and inequality \eqref{coercive-dlc-3}, we obtain
\begin{align}
\label{coercive-dlc-4}
\|\nabla ({\bf u}_{\boldsymbol\varphi})\|_{L_2(\Omega_+\cup\,\Omega_-)^{n\times n}}^2
=|\!|\!|{\bf u}_{\boldsymbol\varphi} |\!|\!|^2
\geq {{\mathcal C}_2^{-2}}\|{\bf u}_{\boldsymbol\varphi} \|_{{\mathcal H}^1(\mathbb R^n\setminus\partial\Omega)^n}^2\,.
\end{align}
Finally, inequalities \eqref{coercive-dlc-5}, \eqref{coercive-dlc-1} and \eqref{coercive-dlc-4} yield the coercivity inequality \eqref{coercive-dlc} with ${\mathcal C}=c_{\mathbb A}^{-1}{\mathcal C}_1^{-1}{\mathcal C}_2^{-2}$. Then the Lax-Milgram lemma implies that operator \eqref{hyp-v-isom} is an isomorphism.
\end{proof}

\section{Transmission problems for the anisotropic Stokes and Navier-Stokes systems with $L_{\infty}$ coefficients. Well-posedness in weighted Sobolev spaces}

The potentials introduced in the previous sections make the analysis of more general transmission problems for Stokes and Navier-Stokes systems rather elementary.

Let us consider the spaces
\begin{align}
\label{sol}
&{\mathfrak A}_p:=({H}^1_{p}(\Omega _+)^n\times L_p(\Omega _+))\times ({\mathcal H}^1_{p}(\Omega _-)^n\times L_p(\Omega _-))\,,\\
\label{data}
&{\mathfrak F}_p:=\widetilde{H}^{-1}_{p}(\Omega _{+})^n\times \widetilde{\mathcal H}^{-1}_{p}(\Omega _{-})^n\times B_{p,p}^{1-\frac{1}{p}}(\partial \Omega )^n\times B_{p,p}^{-\frac{1}{p}}(\partial \Omega )^n\,.
\end{align}

\subsection{Poisson problem of transmission type for the anisotropic Stokes system}

First, for the given data $(\tilde{\bf f}_+,\tilde{\bf f}_-,{\bf h},{\bf g})$ in ${\mathfrak F}_p$, we consider the Poisson problem of transmission type for the anisotropic Stokes system
\begin{equation}
\label{Dirichlet-var-Stokes}
\left\{
\begin{array}{ll}
\partial _\alpha\left(A^{\alpha \beta }\partial _\beta {\bf u}_\pm \right)-\nabla \pi _\pm ={\tilde{\bf f}_{\pm}}|_{\Omega _{\pm }},\
{\rm{div}} \, {\bf u}_\pm =0 & \mbox{ in } \Omega _{\pm },\
\\
{\gamma }_{+}{\bf u}_+-{\gamma }_{-}{\bf u}_{-}={\bf h} &  \mbox{ on } \partial \Omega ,\\
{\bf T}^{+}({\bf u}_+,\pi _+;\tilde{\bf f}_+)-{\bf T}^{-}({\bf u}_-,\pi _- ;\tilde{\bf f}_-)={\bf g} &  \mbox{ on } \partial \Omega .
\end{array}\right.
\end{equation}
The left-hand side in the last transmission condition in \eqref{Dirichlet-var-Stokes} is to be understood in the sense of formal conormal derivatives, cf. Definition~\ref{conormal-derivative-var-Brinkman}. 
\begin{theorem}
\label{T-p}
Let $\mathbb A$ satisfy conditions \eqref{Stokes-1} and \eqref{mu}, $p_*\in (2,\infty )$ be as in Lemma~\ref{lemma-a47-1-Stokes} and $p\in {\mathcal R}(p_*,n)$, cf. \eqref{p*}.
Then for all given data $(\tilde{\bf f}_+,\tilde{\bf f}_-,{\bf h},{\bf g})$ in ${\mathfrak F}_p$, the transmission problem \eqref{Dirichlet-var-Stokes} has a unique solution $(({\bf u}_+,\pi _+),({\bf u}_-,\pi _-))\in {\mathfrak A}_p$. Moreover, there exists a constant $C=C(\partial \Omega , c_{\mathbb A} ,p,n)>0$ such that
\begin{align}
\label{solution-transm}
\|(({\bf u}_+,\pi _+),({\bf u}_-,\pi _-))\|_{{\mathfrak A}_p}\leq C\|(\tilde{\bf f}_+,\tilde{\bf f}_-,{\bf h},{\bf g})\|_{{\mathfrak F}_p}\,.
\end{align}
\end{theorem}
\begin{proof}
Theorem \ref{slp-var-apr-1-p} yields uniqueness.
Now we show existence, by considering the potentials
\begin{align*}
&{\bf u}_\pm =\big(\boldsymbol{\mathcal N}_{\mathbb R^n}\tilde{\bf f}_\pm \big)|_{\Omega _\pm }+{\bf V}_{\partial\Omega}{\bf g}_0-{\bf W}_{\partial\Omega}{\bf h}_0,\
\pi _\pm =\big({\mathcal Q}_{\mathbb R^n}\tilde{\bf f}_\pm \big)|_{\Omega _\pm }+{\mathcal Q}^s_{\partial\Omega}{\bf g}_0-{\mathcal Q}^d_{\partial\Omega}{\bf h}_0 \mbox{ in } \Omega _\pm ,
\\
&{\bf h}_0:={\bf h}-\big\{\gamma _+\big(\big(\boldsymbol{\mathcal N}_{\mathbb R^n}\tilde{\bf f}_+\big)|_{\Omega _+}\big)-\gamma _-\big(\big(\boldsymbol{\mathcal N}_{\mathbb R^n}\tilde{\bf f}_{-}\big)|_{\Omega _-}\big)\big\},\\
&{\bf g}_0:={\bf g}-\big\{{\bf T}^+\big(\big(\boldsymbol{\mathcal N}_{\mathbb R^n}\tilde{\bf f}_+\big)|_{\Omega _+},\big({\mathcal Q}_{\mathbb R^n}\tilde{\bf f}_+\big)|_{\Omega _+};\tilde{\bf f}_+\big)-{\bf T}^{-}\big(\big(\boldsymbol{\mathcal N}_{\mathbb R^n}\tilde{\bf f}_{-}\big)|_{\Omega _-},\big({\mathcal Q}_{\mathbb R^n}\tilde{\bf f}_-\big)|_{\Omega _-};\tilde{\bf f}_-\big)\big\},
\end{align*}
where ${\bf h}_0\in B_{p,p}^{1-\frac{1}{p}}(\partial \Omega )^n$ and ${\bf g}_0\in B_{p,p}^{-\frac{1}{p}}(\partial \Omega )^n$.
According to Definitions \ref{Newtonian-B-var-variable-1}, \ref{s-l-S-variational-variable} and  \ref{d-l-variational-var},  and Lemmas \ref{continuity-sl-S-h-var} and \ref{continuity-dl-h-var} (ii), we deduce that $(({\bf u}_+,\pi _+),({\bf u}_-,\pi _-))$ given above is the unique solution of the transmission problem \eqref{Dirichlet-var-Stokes} in the space ${\mathfrak A}_p$. Moreover, the operator
\begin{align}
\label{solution-oper}
\boldsymbol{\mathcal T}_{(p)}:{\mathfrak F}_p\to {\mathfrak A}_p\,,
\end{align}
which associates to the given data $(\tilde{\bf f}_+,\tilde{\bf f}_-,{\bf h},{\bf g})\in {\mathfrak F}_p$ the unique solution $(({\bf u}_+,\pi _+),({\bf u}_-,\pi _-))\in {\mathfrak A}_p$ of the transmission problem \eqref{Dirichlet-var-Stokes}, is bounded and linear, implying also inequality \eqref{solution-transm}.
\end{proof}

\subsection{Poisson problem with transmission conditions for the anisotropic Stokes and Navier-Stokes systems in $L_p$-based weighted Sobolev spaces}

In this subsection we restrict our analysis to the cases $n=3$ and $n=4$, for which some necessary embedding results hold.
Next, we consider the following Poisson problem of transmission type for the Stokes and Navier-Stokes systems
\begin{equation}
\label{Stokes-NS}
\left\{
\begin{array}{ll}
\partial _\alpha\left(A^{\alpha \beta }\partial _\beta {\bf u}_+\right)-\nabla \pi _+={\tilde{\bf f}_{+}}|_{\Omega _{+}}+\lambda ({\bf u}_+\cdot \nabla ){\bf u}_+\,, \
{\rm{div}} \, {\bf u}_+=0 & \mbox{ in } \Omega _{+},\
\\
\partial _\alpha\left(A^{\alpha \beta }\partial _\beta {\bf u}_-\right)-\nabla \pi _-={\tilde{\bf f}_{-}}|_{\Omega _{-}}\,, \
{\rm{div}} \, {\bf u}_-=0 & \mbox{ in } \Omega _{-},\
\\
{\gamma }_{+}{\bf u}_+-{\gamma }_{-}{\bf u}_{-}={\bf h} &  \mbox{ on } \partial \Omega ,\\
{\bf T}^{+}({\bf u}_+,\pi _+;\tilde{\bf f}_{+}+\mathring{E}_+\left(\lambda ({\bf u}_+\cdot \nabla ){\bf u}_+\right))-{\bf T}^{-}({\bf u}_-,\pi _- ;\tilde{\bf f}_-)={\bf g} &  \mbox{ on } \partial \Omega ,
\end{array}\right.
\end{equation}
with $\mathring{E}_+$ the extension by zero outside $\Omega _+$,
$\lambda \!\in \!L_{\infty }(\Omega _+)$,
and the left-hand side in the last transmission condition in \eqref{Stokes-NS} is to be understood in the sense of formal conormal derivatives, cf. Definition~\ref{conormal-derivative-var-Brinkman}. 
We will show the following result (see \cite[Theorem 5.2]{K-L-M-W} for the Stokes and Navier-Stokes systems in the isotropic constant-coefficient case, $\mathbb A=\mathbb I$.
\begin{theorem}
\label{well-posed-N-S-Stokes} 
Let $n\in\{3,4\}$, $\mathbb A$ satisfy conditions \eqref{Stokes-1} and \eqref{mu}, $\lambda \in L_{\infty }(\Omega _+)$,  and $p_*\in (2,\infty )$ be as in Lemma~\ref{lemma-a47-1-Stokes}. 
Then 
for any $p\in \left(\frac{p_*}{p_*-1},p_*\right)\cap \left[\frac{n}{2},n\right)$ there exist two constants, 
$\zeta _p,\eta _p>0$ depending on ${\Omega }_{+}$, ${\Omega }_{-}$, $\lambda $, $c_{\mathbb A}$, $n$, and $p$, with the property that for all given data $\big({\tilde{\bf f}}_{+},{\tilde{\bf f}}_{-},{\bf h},{\bf g}\big)\in {\mathfrak F}_p$ satisfying the condition
$\|\big({\tilde{\bf f}}_{+},{\tilde{\bf f}}_{-},{\bf h},{\bf g}\big)\|_{{\mathfrak F}_p}\leq \zeta _p,$
the transmission problem \eqref{Stokes-NS} has a unique solution 
$\left(({\bf u}_{+},\pi _{+}),({\bf u}_{-},\pi _{-})\right)\in {\mathfrak A}_p$, such that
$\|{\bf u}_+\|_{H^1(\Omega _+)^n}\leq \eta _p\,.$
\end{theorem}
\begin{proof}
Let
\begin{align}
\label{I-oper}
&\widetilde{\mathcal I}_{\lambda ;{\Omega }_{+}}({\bf v}):=\mathring{E}_+\left(\lambda ({\bf v}\cdot \nabla ){\bf v}\right),\ \forall \, {\bf v}\in H^1_{p}(\Omega _+)^n\,.
\end{align}
The Sobolev embedding Theorem (cf. Theorem 4.12 in \cite{Adams2003}) implies  that for any 
$p\in\left[\frac{n}{2},n\right)$, the embeddings
\begin{align}
H_p^1(\Omega _+)\hookrightarrow L_{\frac{np}{n-p}}(\Omega _+)\hookrightarrow L_{n}(\Omega _+),\ 
H_{p'}^1(\Omega _+)\hookrightarrow L_{\frac{np'}{n-p'}}(\Omega _+)
\end{align}
are continuous, and by duality the last embedding implies that the embedding
\begin{align}
L_{\frac{np}{n+p}}(\Omega _+)\hookrightarrow \widetilde H_{p}^{-1}(\Omega _+)
\end{align}
is also continuous. Applying the H\"{o}lder inequality we then deduce 
\begin{multline}
\label{ineq-2}
\|\mathring{E}_+\left(\lambda {v}{w}\right)\|_{\widetilde H^{-1}_{p}(\Omega_+)}
\le c_0 \|\lambda \|_{L_\infty(\Omega_+)}\|vw\|_{L_{\frac{np}{n+p}}(\Omega_+)}\\
\le c_0\|\lambda \|_{L_\infty(\Omega_+)}\|v\|_{L_n(\Omega_+)}\|w\|_{L_p(\Omega_+)}
\leq c_1\|{v}\|_{H^1_{p}(\Omega _{+})}\|{w}\|_{L_p(\Omega _+)}
\end{multline}
(see also \cite[Lemma 5.1]{K-L-M-W} for $p\!=\!2$, and \cite[Lemma 11.3]{Med-ZAMP-18}).

Therefore, for any ${\bf v}\in H^1_{p}(\Omega _{+})^n$ (and accordingly $\nabla {\bf v}\in L_p(\Omega _+)^{n\times n}$), we obtain that 
$\widetilde{\mathcal I}_{\lambda ;{\Omega }_{+}}({\bf v})\in \widetilde H^{-1}_{p}(\Omega_+)^n$ and
\begin{align}
\label{ineq-1-0}
&
\|\widetilde{\mathcal I}_{\lambda ;{\Omega }_{+}}({\bf v})\|_{\widetilde H^{-1}_{p}(\Omega_+)^n}
\leq c_1\|{\bf v}\|_{H^1_{p}(\Omega _{+})^n}\|\nabla {\bf v}\|_{L_p(\Omega _+)^{n\times n}}
\leq c_1\|{\bf v}\|_{H^1_{p}(\Omega _{+})^n}^2,\\
\label{ineq-3}
&\|\widetilde{\mathcal I}_{\lambda ;\Omega_+}({\bf v})
-\widetilde{\mathcal I}_{\lambda ;\Omega_+}({\bf w})\|_{\widetilde H^{-1}_p(\Omega_+)^n}
\leq c_1\left(\|{\bf v}\|_{H^1_{p}(\Omega _{+})^n}+\|{\bf w}\|_{H^1_{p}(\Omega _{+})^n}\right)
\|{\bf v}-{\bf w}\|_{H^1_{p}(\Omega _{+})^n}\,.
\end{align}
Thus, the nonlinear operator 
$\widetilde{\mathcal I}_{\lambda ;\Omega_+}:H^1_{p}(\Omega _{+})^n\to \widetilde H^{-1}_{p}(\Omega_+)$ 
is continuous  and bounded in the sense of \eqref{ineq-1-0}.

We now construct a nonlinear operator ${\mathcal U}_{(p);+}$ that maps a closed ball ${\bf B}_{\eta_{p}}$ of the space $H^1_{p;\mathrm{div}}({\Omega }_{+})^n$ (of divergence-free vector fields in $H^1_{p}({\Omega }_{+})^n$) to ${\bf B}_{\eta_{p}}$ and is a contraction on ${\bf B}_{\eta_{p}}$. Then the unique fixed point of ${\mathcal U}_{(p):+}$ will determine a solution of nonlinear problem \eqref{Stokes-NS}.

For a fixed $\mathbf u_+\in H^1_{p;\mathrm{div}}(\Omega _{+})^n$, we consider the following linear Poisson problem of transmission type for the Stokes system 
in the unknown $({\bf v}_{+},q_{+})$, $({\bf v}_{-},q_{-})$
\begin{equation}
\label{Newtonian-D-B-F-new2-D-Rn}
\left\{\begin{array}{lll}
\partial _\alpha\left(A^{\alpha \beta }\partial _\beta {\bf v}_+\right)-\nabla q_+={\tilde{\bf f}_{+}}|_{\Omega _{+}}+\left(\widetilde{\mathcal I}_{\lambda ;{\Omega }_{+}}({\bf u}_+)\right)|_{\Omega _{+}}\,, \
{\rm{div}} \, {\bf v}_+=0 & \mbox{ in } \Omega _{+},\
\\
\partial _\alpha\left(A^{\alpha \beta }\partial _\beta {\bf v}_-\right)-\nabla q_-={\tilde{\bf f}_{-}}|_{\Omega _{-}}\,, \
{\rm{div}} \, {\bf v}_-=0 & \mbox{ in } \Omega _{-}\,, \
\\
{\gamma }_{+}{\bf v}_+-{\gamma }_{-}{\bf v}_{-}={\bf h} &  \mbox{ on } \partial \Omega \,,\\
{\bf T}^{+}({\bf v}_+,q_+;\tilde{\bf f}_{+}+\widetilde{\mathcal I}_{\lambda ;{\Omega }_{+}}({\bf u}_+))-{\bf T}^{-}({\bf v}_-,q_- ;\tilde{\bf f}_-)={\bf g} &  \mbox{ on } \partial \Omega \,.
\end{array}\right.
\end{equation}
Since $\big(\tilde{\bf f}_{+}+\widetilde{\mathcal I}_{\lambda ;{\Omega }_{+}}({\bf u}_+)\big)\in \big(H^1_{p'}(\Omega _+)^n\big)'$, Theorem \ref{T-p} implies that problem \eqref{Newtonian-D-B-F-new2-D-Rn} has a unique solution expressed in terms of the linear continuous operator $\boldsymbol{\mathcal T}_{(p)}:{\mathfrak F}_p\to {\mathfrak A}_p$ given by \eqref{solution-oper}, as
\begin{align}
\label{solution-v0}
\left({{\bf v}_+},{q_+},{{\bf v}_-},{q_-}\right)&:=
\left({\mathcal U}_{(p);+}({\bf u}_+),{\mathcal P}_{(p);+}({\bf u}_+),{\mathcal U}_{(p);-}({\bf u}_+),{\mathcal P}_{(p);-}({\bf u}_+)\right)\nonumber\\
&=\boldsymbol{\mathcal T}_{(p)}\big({\tilde{\bf f}}_{+}|_{\Omega _{+}}+\widetilde{\mathcal I}_{\lambda ;\Omega_+}({\bf u}_+)|_{\Omega_+},\ {\tilde{\bf f}}_{-}|_{\Omega _{-}},\ {\bf h},\ {\bf g}\big)\in {\mathfrak A}_p\,.
\end{align}
The nonlinear operator $\widetilde{\mathcal I}_{\lambda ;\Omega_+}:H^1_{p}(\Omega _{+})^n\to (H^1_{p'}(\Omega _+)^n)'$ is continuous as well.
Then by \eqref{ineq-1-0} there exists a constant $c_*=c_*({\Omega }_{+},{\Omega }_{-},n,p,c_{\mathbb A})>0$ such that
\begin{align}
\label{estimate-D-B-F-new1-new1-D-Rn}
\big\|\big({\mathcal U}_{(p);+}({\bf u}_+),&{\mathcal P}_{(p);+}({\bf u}_+),{\mathcal U}_{(p);-}({\bf u}_+),{\mathcal P}_{(p);-}({\bf u}_+)\big)\big\|_{{\mathfrak A}_p}
\nonumber\\
&\leq c_*\big\|\big({\tilde{\bf f}}_{+},{\tilde{\bf f}}_{-},{\bf h},{\bf g}\big)\big\|_{{\mathfrak F}_p}
+c_{*}c_1{\|{\bf u}_{+}\|_{H^1_{p}({\Omega }_{+})^n}^2},
\ \ \forall\, {{\bf u}_{+}\in H^1_{p;\mathrm{div}}({\Omega }_{+})^n}\,.
\end{align}

Next we show that the nonlinear operator ${\mathcal U}_{(p);+}$ has a fixed point 
${\bf u}_{+}\in H_{p;\rm{div}}^{1}({\Omega }_{+})^n$.
Let
\begin{align}
\label{Newtonian-D-B-F-new9n}
\eta _p:=(4c_1c_*)^{-1},\ \zeta _p:=3\eta_p/(4c_*),
\end{align}
and 
$\mathbf{B}_{\eta _p}:=\big\{{\bf v}_{+}\in H^1_{p;\mathrm{div}}({\Omega }_{+})^n:\|{\bf v}_{+}\|_{H^{1}({\Omega }_{+})^n}\leq \eta _p\big\}. $
In addition, assuming that 
\begin{align}
\label{cond-small-sm-msn}
\big\|\big({\tilde{\bf f}}_{+},{\tilde{\bf f}}_{-},{\bf h},{\bf g}\big)\big\|_{{\mathfrak F}_p}\leq \zeta _p,
\end{align}
and using \eqref{estimate-D-B-F-new1-new1-D-Rn}, \eqref{cond-small-sm-msn}, we obtain that ${\mathcal U}_{(p);+}$ maps the closed ball $\mathbf{B}_{\eta _p}$ to itself.

By using expression \eqref{solution-v0} of ${\mathcal U}_{(p);+}$ and inequality \eqref{ineq-3}, we obtain the estimate
\begin{align}
\label{4.30n}
\|{\mathcal U}_{(p);+}({\bf v}_{+})-{\mathcal U}_{(p);+}({\bf w}_{+})
\|_{{H}^1_{p}({\Omega }_{+})^n}
\!\leq \!\frac{1}{2}\|{\bf v}_+\!-\!{\bf w}_+\|_{H^1_{p}({\Omega }_{+})^n},
\end{align}
for all ${\bf v}_+,{\bf w}_+\in \mathbf{B}_{\eta _p}.$
Hence, ${\mathcal U}_{(p);+}:\mathbf{B}_{\eta _p}\to \mathbf{B}_{\eta _p}$ {is a contraction}.
Then the Banach fixed point Theorem yields that ${\mathcal U}_{(p);+}$ has a unique fixed point ${\bf u}_{+}\!\in \!\mathbf{B}_{\eta _p}$, i.e., ${\mathcal U}_{(p);+}({\bf u}_{+})={\bf u}_{+}$, and in view of \eqref{solution-v0}, $\left(({\bf u}_{+},{\mathcal P}_{(p);+}({\bf u}_{+})),({\mathcal U}_{(p);-}({\bf u}_{+}),{\mathcal P}_{(p);-}({\bf u}_{+}))\right)$ determines a solution of the nonlinear problem \eqref{Stokes-NS} in the space ${\mathfrak A}_p$, which is unique, due to an argument similar to that in the proof of \cite[Theorem 5.2]{K-L-M-W}.
\end{proof}


\section{Auxiliary results: Equivalent norms in Banach spaces}

The next result plays a major role in establishing the equivalence of norms on Banach spaces, in particular, on some Sobolev spaces that appear in our arguments (cf. \cite[Lemma 11.1]{Tartar}).
\begin{lemma}
\label{Tartar-lemma}
Let $(X,\|\cdot \|_X)$ be a Banach space, and let $(Y,\|\cdot \|_Y)$, $(Z,\|\cdot \|_Z)$, $(\Upsilon ,\|\cdot \|_\Upsilon )$ be normed spaces. Let ${\mathcal P}:X\to Y$, ${\mathfrak C}:X\to Z$ and ${\mathcal T}:X\to \Upsilon $ be linear and continuous operators, such that  
\begin{itemize}
\item[$(i)$]
The operator ${\mathfrak C}:X\to Z$ is compact.
\item[$(ii)$]
$\|P(\cdot )\|_{Y}+\|{\mathfrak C}(\cdot )\|_{Z}$ is a norm on $X$ equivalent to the norm $\|\cdot \|_X$.
\item[$(iii)$]
The operator ${\mathcal T}:X\to \Upsilon $ satisfies the condition ${\mathcal T}(u)\neq 0$ whenever $P(u)=0$ and $u\neq 0$.
\end{itemize}
Then the mapping $|\!|\!|\cdot |\!|\!|:X\to {\mathbb R}_+$ given by
\begin{align}
|\!|\!|u|\!|\!|:=\|P(u)\|_Y+\|{\mathcal T}(u)\|_{\Upsilon },\ u\in X,
\end{align}
is a norm on $X$ equivalent to the given norm $\|\cdot \|_X$.
\end{lemma}
\begin{lemma}
\label{equiv-norm-Sobolev}
Let $\mathbb A$ satisfy conditions \eqref{Stokes-1} and \eqref{mu}. 
Then the formula
\begin{align}
\label{coercive-dlc-apend}
|\!|\!|{\bf v}|\!|\!|^2:=\|\nabla ({\bf v})\|_{L_2(\Omega_+\cup\,\Omega_-)^{n\times n}}^2
+\left|\int_{\partial \Omega }[\gamma {\bf v}]d\sigma \right|^2, \ 
\forall \, {\bf v}\in {\mathcal H}^1({\mathbb R}^n\setminus \partial \Omega )^n
\end{align}
defines a norm on the weighted Sobolev space ${\mathcal H}^1({\mathbb R}^n\setminus \partial \Omega )^n$, which is equivalent to the norm 
\begin{align}
\label{standard-weight}
\|{\bf v}\|_{{\mathcal H}^1({\mathbb R}^n\setminus \partial \Omega )^n}^2
=\|\rho ^{-1}{\bf v}\|_{L_2({\mathbb R}^n)^n}^2
+\|\nabla {\bf v}\|_{L_2(\Omega_+\cup\,\Omega_-)^{n\times n}}^2\,.
\end{align}
\end{lemma}
\begin{proof}
First, we note that $\|\nabla (\cdot )\|_{L_2(\Omega _{-})^{n\times n}}$ is a norm on ${\mathcal H}^1(\Omega _{-})^n$, equivalent to the norm $\|\cdot \|_{{\mathcal H}^1(\Omega _{-})^n}$, defined as in \eqref{standard-weight} with $\Omega_{-}$ in place of ${\mathbb R}^n$ and $\Omega_+\cup\,\Omega_-$ (see, e.g., \cite[Proposition 2.7]{Sa-Se} in the case $n=3$).
Therefore,
\begin{align}
\label{standard-weight-1}
\|\nabla ({\bf v})\|_{L_2(\Omega_-)^{n\times n}}+\|\nabla ({\bf v})\|_{L_2(\Omega_+)^{n\times n}}
+\|{\bf v}\|_{L_2(\Omega_+)^n}
=\|\nabla ({\bf v})\|_{L_2(\Omega_+\cup\,\Omega_-)^{n\times n}}+\|{\bf v}\|_{L_2(\Omega _+)^{n}} 
\end{align}
is a norm on the space ${\mathcal H}^1({\mathbb R}^n\setminus \partial \Omega )^n$, equivalent to the norm \eqref{standard-weight} of this space.

Now, we consider the Banach spaces 
$X:={\mathcal H}^1({\mathbb R}^n\setminus \partial \Omega )^n$, $Y=L_2(\Omega_+\cup\,\Omega_-)^{n\times n}$, $Z=L_2(\Omega _+)^{n}$ 
and 
$\Upsilon :={\mathbb R}^{n}$. 
Also let us consider the operators 
\begin{align}
&P:{\mathcal H}^1({\mathbb R}^n\setminus \partial \Omega )^n\to L_2(\Omega_+\cup\,\Omega_-)^{n\times n},\ P({\bf v})
:=\nabla {\bf v}\,,\\
&{\mathfrak C}:{\mathcal H}^1({\mathbb R}^n\setminus \partial \Omega )^n\to L_2(\Omega _+)^{n},\ {\mathfrak C}({\bf v})
:={\bf v}|_{\Omega _+},\\
&{\mathcal T}:{\mathcal H}^1({\mathbb R}^n\setminus \partial \Omega )^n\to {\mathbb R}^{n},\ {\mathcal T}({\bf v})
:=\int_{\partial \Omega }[\gamma {\bf v}]d\sigma ,
\end{align}
all of them being linear and continuous. Moreover, the operator ${\mathfrak C}$ is compact due to the compact embedding of the space $H^1(\Omega _+)^{n}$ in $L_2(\Omega _+)^{n}$, and the norm in \eqref{standard-weight-1} can be written as
\begin{align}
\label{standard-weight-2}
\|\nabla {\bf v}\|_{L_2(\Omega_+\cup\,\Omega_-)^{n\times n}} + \|{\bf v}\|_{L_2(\Omega _+)^{n}}
=\|P({\bf v})\|_{L_2(\Omega_+\cup\,\Omega_-)^{n\times n}} + \|{\mathfrak C}({\bf v})\|_{L_2(\Omega _+)^{n}}\,.
\end{align}
In addition, the operator ${\mathcal T}$ satisfies the condition ${\mathcal T}({\bf v})\neq 0$ whenever $P({\bf v})=0$ and ${\bf v}\neq {\bf 0}$. Indeed, the condition $P({\bf v})=0$ and ${\bf v}\neq {\bf 0}$ is equivalent to ${\bf v}|_{\Omega _{-}}={\bf 0}$, ${\bf v}|_{\Omega _{+}}={\bf c}\in {\mathbb R}^n$ with ${\bf c}\neq {\bf 0}$.
Assume that ${\mathcal T}({\bf v})=0$. 
Then
\begin{align}
\label{standard-weight-3}
\int_{\partial \Omega }[\gamma {\bf v}]d\sigma =0\,.
\end{align}
Since $[\gamma {\bf v}]={\bf c}$ on $\partial \Omega $, condition \eqref{standard-weight-3} implies that ${\bf c}={\bf 0}$, which contradicts the assumption ${\bf v}\neq {\bf 0}$. Thus, ${\mathcal T}({\bf v})\neq 0$ whenever $P({\bf v})=0$ and ${\bf v}\neq {\bf 0}$, as asserted.

Consequently, the conditions in Lemma \ref{Tartar-lemma} are satisfied, and hence
\begin{align}
\label{coercive-dlc-apend-3}
\|P({\bf v})\|_{Y}+\|{\mathcal T}({\bf v})\|_{\Upsilon }
=\|\nabla {\bf v}\|_{L_2({\mathbb R}^n\setminus \partial \Omega )^{n\times n}}
+ \left|\int_{\partial \Omega }[\gamma {\bf v}]d\sigma \right|
\end{align}
is a norm on ${\mathcal H}^1({\mathbb R}^n\setminus \partial \Omega )^n$ equivalent to the norm $\|\cdot \|_{{\mathcal H}^1({\mathbb R}^n\setminus \partial \Omega )^n}$. This result and the equivalence of the norms \eqref{coercive-dlc-apend} and \eqref{coercive-dlc-apend-3}  
show that \eqref{coercive-dlc-apend} is also a norm in ${\mathcal H}^1({\mathbb R}^n\setminus \partial \Omega )^n$ equivalent to the norm \eqref{standard-weight}.
\end{proof}

\section*{\bf Acknowledgements}
The research has been partially supported by the grant EP/M013545/1: "Mathematical Analysis of Boundary-Domain Integral Equations for Nonlinear PDEs" from the EPSRC, UK.
Mirela Kohr has been also partially supported by the Babe\c{s}-Bolyai University grant AGC35124/31.10.2018.

\end{document}